\documentclass[11pt,twoside,a4paper]{amsart}
\usepackage{color}
\usepackage{srcltx}
\usepackage{fancybox}
\usepackage{fancyhdr}
\usepackage{latexsym}
\usepackage{amsfonts}
\usepackage{amscd}
\usepackage{epsfig}
\usepackage[cp850]{inputenc}
\usepackage[english]{babel}
\usepackage{ifthen}
\usepackage{graphicx}
\usepackage{float}
\usepackage{amsthm}
\usepackage{amsmath}
\usepackage{amssymb}
\usepackage{mathrsfs}
\usepackage[mathscr]{eucal}
\usepackage{dsfont}
\usepackage{amstext}
\usepackage{syntonly}

\usepackage{tikz} 
\usepackage{color}
\usepackage{graphicx}

  \usepackage{enumerate}
\theoremstyle{plain}
\newtheorem{theorem}{Theorem}
\newtheorem{proposition}[theorem]{Proposition}
\newtheorem{lemma}[theorem]{Lemma}
\newtheorem{corollary}[theorem]{Corollary}

\theoremstyle{remark}
\newtheorem{remark}[theorem]{Remark}

\theoremstyle{definition}
\newtheorem{definition}[theorem]{Definition}

\newtheorem{example}[theorem]{Example}

\usepackage{color}

\usepackage{tikz}
\usepackage{color}
\usepackage{graphicx}

\begin{document}

 %\authors{G. Bosi, A. Estevan, M. Zuanon}{Partial representations of orderings}

 %%%%%%%%%%%%%%%%%%%%% Publisher's Area please ignore %%%%%%%%%%%%%%%
%

%%%%%%%%%%%%%%%%%%%%%%%%%%%%%%%%%%%%%%%%%%%%%%%%%%%%%%%%%%%%%%%%%%%%

 \title{PARTIAL REPRESENTATIONS OF ORDERINGS}

% \author{G. BOSI, A. ESTEVAN, MAGALI}%\footnote{Corresponding author.}
% }
% \address{Departamento de Matem\'aticas, Universidad P\'ublica de Navarra,
%Campus Arrosad\'{\i}a\\ Pamplona, 31006, Spain\\ asier.mugertza@unavarra.es}

% \author{J. GUTI\'ERREZ GARC\'IA}
% \address{Departamento de Matemticas, Universidad del Pa\'s Vasco UPV/EHU, Apdo. 644\\
% Bilbao, 48080, Spain\\
% javier.gutierrezgarcia@ehu.es}

% \author {E. INDUR\'AIN%\footnote{Corresponding author.}
 %}
% \address{Departamento de Matem\'aticas, Universidad P\'ublica de Navarra,
%Campus Arrosad\'{\i}a\\ Pamplona, 31006, Spain\\
%steiner@unavarra.es}

\maketitle

 \author{G. BOSI \footnote{Dipartimento de Scienze Economiche, Aziendali, Mathematiche e Statistiche\\ Universit\`a degli Studi di Trieste \\ Piazzale Europa 1, I-34127 \\ Trieste, Italia.% \\
%GIANNI.BOSI@deams.units.it
} }

 \author{A. ESTEVAN \footnote{Departamento de Matem\'aticas, Universidad P\'ublica de Navarra,
Campus Arrosad\'{\i}a\\ Pamplona, 31006, Spain\\ asier.mugertza@unavarra.es}}

 \author{M. ZUANON\footnote{Dipartimento di Economia e Management, Universit\`a degli
Studi di Brescia, Contrada Santa Chiara 50, 25122 Brescia, Italy.}}

 \begin{abstract}
In the present paper a new concept of representability is introduced, which can be applied to not total and also to intransitive relations (semiorders in particular).
This  idea tries to represent the orderings in the simplest manner, avoiding any unnecessary information. For this purpose, the new concept of representability is developed by means of partial functions, so that other common definitions of representability (i.e. (Richter-Peleg) multi-utility, Scott-Suppes representability,...) are now particular cases in which the partial functions are actually functions.  The paper also presents a collection of examples and propositions showing the advantages of this kind of representations, particularly in the case of partial orders and semiorders, %{\color{red}
 as well as some results showing the connections between distinct kinds of representations.%}.
 \end{abstract}
\textbf{keywords:} {partial representability; multi-utility; preorders; semiorders; intransitivity.}

  \maketitle
 \setcounter{footnote}{0}
 \renewcommand{\thefootnote}{\arabic{footnote}}
 \section{Introduction and motivation} \label{s1}

Different kinds of representations of preferences have been recently proposed in the literature in order to consider general situations when completeness is not required. It is well known that in this case more than one function must be used. In some sense, the best way  of representing transitive preferences which are not necessarily total is to invoke  the multi-utility approach, since it provides a characterization of them. In this paper we generalize this classical approach in order to allow nontransitivity of the preferences. This is done by simply allowing each function to be defined not on the whole space of the alternatives but on some subset of it.%\\ \indent

 Given a preorder $\precsim $ on $X$, a real function $u\colon X \to \mathbb{R}$ is said to be \emph{isotonic} or \emph{increasing} if for every $x ,y \in X$ the implication $x\precsim y \Rightarrow u(x)\leq u(y)$ holds true. In addition, if it also holds true that $x\prec y $ implies $u(x)< u(y)$, then $u$ is said to be a \emph{Richter-Peler utility representation}.% or an \emph{order preserving} function.
 
In the  case of  a total preorder $\precsim$ on $X$, it is said to be \emph{representable} \rm if there is a real-valued function $u\colon X\to \mathbb R$ that is \emph{strictly isotonic} or \emph{strictly increasing} (also known as {\em order-preserving}), so that, for every $x, y \in X$, it holds that $x \precsim y \iff u(x) \leq u(y)$. The map $u$ is said to be an \emph{order-monomorphism} \rm (also known as a \emph{utility function} \rm for $\precsim$).

 A (not necessarily total) preorder $\precsim$ on a set $X$   is said to have a {\em
multi-utility representation}  if  there exists a family $\mathcal{U}$ of isotonic real functions such that 
for all points $x ,y \in X$  the equivalence
\begin{equation} \label{mult1} 
x \precsim y \Leftrightarrow  \forall u \in {\mathcal U} \,\,(u(x) \leq u(y))\end{equation}
%$ x \precsim y \Leftrightarrow  \forall u \in {\mathcal U} \,\,u(x) \leq u(y)$ %\end{equation}
 holds. 
This kind of representation, whose main feature is to fully characterize
the preorder, was first introduced by Vladimir L. Levin \cite{Levin} (see also \cite{Levin2}), who called \emph{functionally closed}
a preorder admitting a multi-utility representation. However, the first systematic study of multi-utility representations is due to Ozgur Evren and Efe A. Ok \cite{EvO}, who presented different conditions for the existence of continuous multi-utility representations.

% A preorder $\precsim$ on a set $X$   is said to have a {\em multi-utility representation}  if for all points $x ,y \in X$  the equivalence
%\begin{equation} \label{mult1} x \precsim y \Leftrightarrow  \forall u \in {\mathcal U} \,\,(u(x) \leq u(y)\end{equation}
%holds.
 While a multi-utility representation  exists for every not necessarily total preorder $\precsim$ on $X$ (see Evren and Ok \cite[Proposition 1]{EvO}), its application is in some sense limited, since if we start from a binary relation $\precsim$ on set $X$ and it admits the representation above, then it must be necessarily a preorder  (i.e., a reflexive and transitive binary relation). Nevertheless, it is interesting to search for a {\em continuous multi-utility representation} of a preorder $\precsim$ when the set $X$ is endowed with a topology $\tau$  (cf., for instance, Evren and Ok \cite{EvO}, Bosi and Herden \cite{BH} and Alcantud et al. \cite{Alc}). The existence of a finite multi-utility representation was studied by Ok \cite{Ok} and Kaminski \cite{Kam}, who refers to {\em representation by means of multi-objective functions}.

We recall that a particular case of the previous representation is the so called {\em Richter-Peleg multi-utility representation} (see Minguzzi \cite{Min2}), which holds when all the functions of the family $\mathcal U$ in representation (\ref{mult1}) are {\em order-preserving} for the preorder $\precsim$ (i.e., for all $u \in {\mathcal U}$, and $x,y \in X$, $x \prec y$ implies that $u(x) < u(y)$). It is well known that in this case the family $\mathcal U$ also represents the {\em strict part} $\prec$ of $\precsim$ (see Alcantud et al. \cite[Remark 2.3]{Alc}), in the sense that, for all $x,y \in X$, $x \prec y$ if and only if $u(x) < u(y)$ for all $u \in {\mathcal U}$.

Therefore if we want to represent a binary relation $\precsim$ which is reflexive and not necessarily transitive (like {\em interval orders} or {\em semiorders}, for example), we cannot use the multi-utility approach. In order to remove this restriction, Nishimura and Ok  \cite{Nish} introduced very recently the following representation of a necessarily reflexive binary relation $\precsim$, which allows intransitivity: for a set $\Bbb U$  of sets $\mathcal U$ of real-valued functions $u$ on $X$, and all points $x,y \in X$, \begin{equation} \label{equdef} x \precsim y \Leftrightarrow \sup_{{\mathcal U} \in {\Bbb U}}\inf_{u \in {\mathcal U}}\left( u(y) - u(x) \right) \geq 0.
\end{equation}

While this {\em maxmin multi-utility representation} fully characterizes the binary relation $\precsim$ in the general case when it is neither total nor transitive, we can consider that this latter representation is much demanding and difficult to perform, since it requires, in some sense, a lot of {\em information}, represented by a set of sets of real-valued functions.

In this paper,  we introduce the concept of \emph{partial multi-utility representation} of a preorder as a coherent and  practical representation that characterizes the order structure. This means that we refer to the multi-utility approach (\ref{mult1}) in the much more general case when the functions $u$ are not required to be defined on the whole set $X$, but they are only {\em partial} (in the sense that, generally speaking, they are defined on subsets of $X$).

Needless to say, this generalization leads to a new representation which is compatible both with incompleteness and intransitivity. Indeed, the characteristic feature of the present work is to allow intransitivity in a multi-utility fashion.

Referring to the multi-utility representation (\ref{mult1}), usually it is not easy or even possible (when continuity or at least upper semicontinuity of the functions is required) to construct a representation of a binary relation through functions that assign a value for each element of the set. Actually, since these relations fail to be total in general, it has not too much sense to impose a value to each element by each function. So, it seems consistent to provide the representation a degree of uncertainty or `undefinition': if we cannot compare a pair $(x,y)$, maybe we can avoid mapping $x$ and $y$ with each function of the representation (i.e., with each function $u \in {\mathcal U}$).

This  uncertainty or `undefinition'  allows us to construct representations more easily (even when the order structure is not representable in the usual manner), and on the other hand it  facilitates the continuity of the representation.

Although multi-utility representations deserve their interest in economics for the aforementioned reasons, they also appear in computer science, especially in {\em distributed systems} (see e.g  Lamport \cite{Lamport}, Estevan \cite{DS}, Fidge \cite{fid}, Raynal and Singhal \cite{ray} and Mattern \cite{virtual}) and even in physics (see e.g. Panangaden \cite{Panan}).
%{\color{red}
 In computer science the terminology (\emph{labelings, random structures, clocks, ...}) is quite different with respect to the field of economics, but the ideas are essentially the same. Moreover, the idea of partial functions is not strange at all in computation so, this new theory may be quite useful for this field.
 %}

 \emph{The structure of the paper goes as follows}\rm:
After the introduction and the motivation, a section of notation and preliminaries is included. In Section~\ref{s3} the new concept of \emph{partial representability} is introduced, as well as some examples and propositions showing its advantages for the case of preorders. %{\color{red} 
In this section, we also %reformulate  some results from \cite{Alc} (but, now, for partial representability)  showing the
show some connections between distinct kinds of representations. %}.
Finally, in Section~\ref{s4} we focus our attention on the usefulness of the partial representability for intransitive relations, and in particular, we deepen  the study of semiorders presenting a partial version of the Scott-Suppes representation. A Section~\ref{s6} of further comments closes the paper. %{\color{green}
 There, incidentally, it is shown that a Theorem of Evren and Ok \cite{EvO} is incorrect.%}
 \setcounter{footnote}{0}
 \renewcommand{\thefootnote}{\arabic{footnote}}

  \setcounter{footnote}{0}
 \renewcommand{\thefootnote}{\arabic{footnote}}

 \section{Notation and preliminaries} \label{s2}

 From now on  $X$ will denote  a nonempty set.
 
  \begin{definition}
   A binary relation $\mathcal{R}$ on $X$ is a subset of the Cartesian product $X \times X$. Given two elements $x,y \in X$, we will use the standard notation $x \mathcal{R} y$ to express that the pair $(x,y)$ belongs to $\mathcal {R}$.

 Associated to a binary relation $\mathcal{R}$ on a set $X$, we consider its \emph{negation} \rm (respectively, its %\emph{transpose} \rm or
  \emph{dual}\rm) as the binary relation $\mathcal{R}^c$ (respectively, $\mathcal{R}^t$) on $X$ given by $(x,y) \in \mathcal{R}^c \iff (x,y) \notin \mathcal{R}$ for every $ x,y \in X$ (respectively, given by $(x,y) \in \mathcal{R}^t \iff (y,x) \in \mathcal{R}, \ $ for every $ x,y \in X)$. We also define the %\emph{adjoint} \rm or
   \emph{codual} \rm $\mathcal{R}^a$ of the given relation $\mathcal{R}$, as $\mathcal{R}^a = (\mathcal{R}^t)^c$.

 A binary relation $\mathcal{R}$ defined on a set $X$ is said to be:
 \begin{itemize}
 \item[(i)] \emph{reflexive} \rm if $x \mathcal{R} x$ holds for every $x \in X$,
 \item[(ii)] \emph{irreflexive} \rm if $x \mathcal{R}^c x$ holds for every $x \in X$,
 \item[(iii)] \emph{symmetric} \rm if $\mathcal{R}$ and $\mathcal{R}^t$ coincide,
 \item[(iv)] \emph{antisymmetric} \rm if $\mathcal{R} \cap \mathcal{R}^t \subseteq \Delta = \{ (x,x): x \in X \}$,
 \item[(v)] \emph{asymmetric} \rm if $\mathcal{R} \cap \mathcal{R}^t = \varnothing$,
 \item[(vi)] \emph{total} \rm if $\mathcal{R} \cup \mathcal{R}^t = X \times X$,
 \item[(vii)] \emph{transitive} \rm if $x \mathcal{R} y $ and $ y \mathcal{R} z \Rightarrow x \mathcal{R} z $ for every $x,y,z \in X$.
 \end{itemize}
 \end{definition}

 In the particular case of a nonempty set where some kind of \emph{ordering} \rm (e.g., preorder, interval order,  biorder, etc.) has been defined, the standard notation is different. We include it here for sake of completeness, and we will use it throughout the present manuscript.

In what follows   $\precsim$  denotes a reflexive binary relation on $X$. 
Given a reflexive binary relation $\precsim$, then as usual we denote the associated \emph{asymmetric} relation by $\prec$ and the
associated \emph{indifference} relation by $\sim$ and these are defined, respectively, by $[x \prec y \iff (x \precsim y) \text{ and } (y \precsim^c x)]$ and $[x \sim y \iff (x\precsim y) \text{ and } (y \precsim x)]$.
 If two elements are not comparable, that is, if it holds true that $x\precsim^c y$ as well as $y\precsim^c x$ for some $x,y\in X$, then we shall denote that by $x \bowtie y.$

 \begin{definition}
  A \emph{preorder} $\precsim$ on $X$ is a binary relation on $X$ which is reflexive and transitive.
 An antisymmetric preorder is said to be an \emph{order}. A \emph{total preorder} \rm $\precsim$ on a set $X$ is a preorder such that if $x,y \in X$ then $(x \precsim y) \text{ or } (y \precsim x)$. A total order is also called a \emph{linear order}\rm, and a totally ordered set $(X,\precsim)$ is also said to be a \emph{chain}\rm. Usually, an order that fails to be total is also said to be a \emph{partial order}.
 If $\precsim$ is a preorder on $X$, then %as usual we denote the associated \emph{asymmetric} relation by $\prec$ and 
 the associated indifference  relation  $\sim$ is actually an equivalence relation.
%and these are defined, respectively, by $[x \prec y \iff (x \precsim y) \wedge\neg(y \precsim x)]$ and $[x \sim y \iff (x\precsim y) \wedge (y \precsim x)]$.
The asymmetric part of a linear order (respectively, of a total preorder) is said to be a \emph{strict linear order} \rm (respectively, a strict total preorder).
Usually, in the case of partial orders (in particular dealing on finite sets), the relation $\precsim$ is also denoted by $\sqsubseteq$, and the corresponding strict part by $\sqsubset$.

%{\color{red}
In case of not total relations defined on a set $X$, if one element is not related or comparable to any other  of the set, this element is said to be an \emph{isolated point} \cite{BRME}. Given a preorder $\precsim$ on $X$, a set $Y\subseteq X$ is called an \emph{antichain} if $\prec_{|Y}=\emptyset$. The \emph{width} (denoted by $w(X,\precsim)$) of a preordered set is the cardinality of the largest antichain $Y$ contained in $X$. A partial order $\precsim$ on $X$ is \emph{near-complete} if and only if
 $w(X,\precsim)< \infty $.

For every $x \in {X}$ we define the following subsets of ${X}$:
\[l(x)= \{y \in {X} \mid y \prec x\},\,\,\,r(x)= \{z \in {X} \mid x
\prec z\},\] \[d(x)= \{y \in {X} \mid y \precsim x\},\,\,\,i(x)= \{z
\in {X} \mid x \precsim z\}.\]

%}

 \end{definition}

\begin{definition}
A preorder $\precsim$ on a topological space $(X, \tau )$ is \emph{regular} if and only if for each $x\in X$ sets $i(x)$ and $d(x)$ are closed.
\end{definition}

 Next Definition~\ref{lkeste} introduces the notion of representability\footnote{The notion of representability of some orderings is not unique (see \cite{ales,Naka,gurea1}).}. The idea behind representability corresponds to the possibility of converting qualitative scales (preferences) into quantitative ones, comparing real numbers instead of, just, elements of a nonempty set.

 \begin{definition} \label{lkeste} A total preorder $\precsim$ on $X$ is called \emph{representable} \rm if there is a real-valued function $u\colon X\to \mathbb R$ that is order-preserving, so that, for every $x, y \in X$, it holds that $[x \precsim y \iff u(x) \leq u(y)]$. The map $u$ is said to be an \emph{order-monomorphism} \rm (also known as a \emph{utility function} \rm for $\precsim$).

\end{definition}

Now we introduce the definitions of some  intransitive relations, namely interval orders and semiorders \cite{ales,BRME,Fis2,Luc,Sco}. 

 \begin{definition} An {\em  interval order} $\precsim$ on a set $X$ is a reflexive binary relation on $X$ which in addition satisfies the following condition
for all $x,y, z,w \in X$:
\[ (x\precsim z) \,\,and \,\, (y \precsim
w) \Rightarrow (x\precsim w) \,\,or\,\,(y \precsim z).\]

A \emph{semiorder} $\precsim$ on a  set $X$ is a binary relation on $X$ which is an interval order and in addition verifies the following condition for all $x,y, z,w \in X$: \[ (x\precsim y) \,\,and \,\,(y \precsim z) \Rightarrow (x\precsim w) \,\,or \,\,(w \precsim z).\] \end{definition}

\begin{definition} \label{regular} An asymmetric  relation $\prec$ defined on a set $X$ is called \emph{regular with respect to sequences}
if there are no $x,y \in X$, and sequences $(x_n)_{n \in \mathbb{N}}, (y_n)_{n \in \mathbb{N}} \subseteq X$, such that $x \prec \ldots \prec x_{n+1} \prec x_n \prec \ldots \prec x_1$ happens, or, dually, $y_1 \prec \ldots \prec y_n \prec y_{n+1} \prec \ldots \prec y$ holds. %(In other
%words, the set $X$ has no infinite up or down chains as regards $\prec$
%with an upper bound $y$ or a lower bound $x$, respectively).

\end{definition}

\begin{definition} \label{veinte}
A semiorder $\precsim$ defined on a nonempty set $X$, is said to be \emph{regular} (with respect to sequences) if its corresponding strict preference $\prec$ is regular with respect to sequences in the sense of Definition~\ref{regular}. \end{definition}

\begin{definition}A semiorder $\precsim$ \rm defined on $X$ is said to be \emph{SS-representable} (that is {\em representable in the sense of Scott and Suppes})
if there exists a real-valued map $u\colon X \to \mathbb R$ (now again called a \emph{utility function}\rm) such that $x \precsim y \iff u(x)\leq u(y)+1 \ (x,y \in X)$ (see Scott and Suppes \cite{Sco}).

 (In this case, the pair $(u,1)$ is said to be a \emph{Scott-Suppes representation} \rm of the semiorder $\prec$).
\end{definition}

\begin{remark} It is well known that regularity is a necessary condition for the existence of a Scott-Suppes representation (see e.g. \cite{gurea2}).\end{remark}

\section{Partial representability of preorders}\label{s3}

In the present section we introduce the main definitions and some propositions illustrating the advantages of dealing with partial representability instead of the usual representability. Now we focus our attention on preorders, leaving a further study concerning intransitive relations to the next section. %{\color{red}
 %In particular, 
 We also recover some results (Proposition~\ref{Pmurprpmu} and Proposition~\ref{propositiontotal}) introduced in \cite{Alc}, but now adapting and proving them for this new theory of partial representability.%} proposition}\label{Pmurprpmu

The knowledge on preorders, in particular in finite partially ordered sets, is a strong tool in computer sciences. That is the reason why we include some notions (e.g. labelings, random structure, ... \cite{L1}) related to this field. Furthermore, these concepts are strongly related to similar ideas on economics (e.g. utilities, Richter-Peleg multi-utility representations, ... \cite{Alc,BH,EvO,Peleg,Richter}). Besides, partial functions are  quite common in computing so, dealing with partial functions in order to represent orderings could be a good technique.

In the following lines we introduce some basic definitions.

%That is, it seems coherent that if two elements are not comparable, maybe there is a (partial) function
\begin{definition}\label{partialfunction}
A \emph{partial function} from $X$ to $Y$ (written as $f\colon X \nrightarrow Y$) is a function $f\colon X' \to Y$, for some subset $X'$ of $X$. It generalizes the concept of a function $f\colon X \to Y$ by not forcing $f$ to map every element of $X$ to an element of $Y$ (only some subset $X'$ of $X$). If $X' = X$, then $f$ is called a total function and is equivalent to a function.
\end{definition}

\begin{definition}\label{partialfunctioncont}
A partial function $f\colon (X,\tau_X) \nrightarrow (Y,\tau_Y)$ is \emph{continuous} if the  corresponding total function $f\colon (X',\tau_{X'}) \to (Y,\tau_Y)$ is continuous (here $X'$ is the biggest subset  of $X$ where the partial function $f$ is defined and $\tau_{X'} $ is the topology on $X'$ inherited from $(X,\tau_X)$).
\end{definition}

\begin{definition}\label{partialrep}
Let $\precsim$ be a preorder on a nonempty set $X$. We will say that the preorder is \emph{partial multi-utility representable} (or \emph{multi-utility representable through partial functions}) if there exists a family of real partial functions $\mathcal{U}$ on $X$ such that for any pair $x\precsim y$
there exists $u\in \mathcal{U}$ such that  $u(x)\leq u(y)$, as well as  $v(x)\leq v(y)$  for any  $v\in \mathcal{U} $ which is defined on both $x$  and  $y$.
\end{definition}

\begin{remark}
With definition above, notice that:

\noindent $(i)$ $x\prec y$ if and only if there exists $u\in \mathcal{U}$ such that  $u(x)< u(y)$, as well as  $v(x)\leq v(y)$  for any  $v\in \mathcal{U} $ which is defined on both $x$  and  $y$.
\medskip

\noindent $(ii)$ $x\sim y$ if and only if there exists $u\in \mathcal{U}$ such that  $u(x)= u(y)$, as well as  $v(x)= v(y)$  for any  $v\in \mathcal{U} $ which is defined on both $x$  and  $y$.
\medskip

\noindent $(iii)$ $x\bowtie y$ if and only if there exists $u,v\in \mathcal{U}$ such that  $u(x)< u(y)$ and  $v(x)> v(y)$, or there is no    $v\in \mathcal{U} $  defined on both $x$  and  $y$.
\end{remark}

\begin{definition}\label{partialRPrep}
Let $\precsim$ be a preorder on a nonempty set $X$. We will say that the preorder is \emph{partial Richter-Peleg multi-utility representable} (or \emph{Ritcher-Peleg multi-utility representable through partial functions}) if there exists an \emph{isotonic}  partial multi-utility representation $\mathcal{U}$ on $X$. This means that  there exists a  partial multi-utility representation $\mathcal{U}$ on $X$ such that for any pair $x\prec y$
there exists $u\in \mathcal{U}$ such that  $u(x)< u(y)$, as well as  $v(x)< v(y)$  for any  $v\in \mathcal{U} $ which is defined on both $x$  and  $y$.
\end{definition}

\begin{remark}
With definition above, notice that the corresponding indifference and incomparability are characterized like in the previous case of partial multi-utility.
\end{remark}

\begin{remark}\label{remark123}

\noindent$(1)$ In the definitions above we say `partial' because the representation is made by means of partial functions, but this kind of representations totally characterizes the order structure.

\noindent$(2)$
If the partial functions of the definition above are functions (that is, all of them are defined on all the set $X$) then we recover the definition of (Richter-Peleg) multi-utility representation. That is, the partial representation generalizes the concept of representability.

\noindent$(3)$
Notice that, given a partial (Richter-Peleg) multi-utility representation and an element $x$ of the set, if there is no function defined on this point $x$, then this element is an isolated  point.
%These new partial utilities could be useful for dealing with random structures (see Schellekens \cite{L1}) too.  Given a partial order on a finite set $X$ of $n$ elements, the family of all possible utilities $u \colon X \to \{1,2,...,n\}$ (named \emph{labelings}) is known as the \emph{random structure} of the partial order.  These structures are studied in computation, with several questions unsolved yet. So now, we could propose an analogous study with the \emph{partial random structure}, which would be the family of all possible partial utilities.

\end{remark}

%%%%%%%%%%%%%%%%%%%%%%%%%%%%%%%%%%%%%%%%%%%%%%%%%%%%%%%%%%%%%%%%%%%%%%%%%%%%%

%%%%%%%%%%%%%%%%%%%%%%%%%%%%%%%%%%%%%%%%%%%%%%%%%%%%%%%%%%%%%%%%%%%%%%%%%%%%%
%{\color{red}

\begin{proposition}\label{Pmurprpmu}
Let $\precsim$ be a preorder on $X$ and $A$ the set of all isolated points. 
Assume that there exists a (continuous)  partial multi-utility representation $\{v_i\}_{i\in I}$. If there exists a (continuous) Richter-Peleg utility of the preorder on a subset ${Y}$ such that $X\setminus A \subseteq Y$, then there exists a (continuous) partial Richter-Peleg multi-utility representation. 
\end{proposition}
\begin{proof}
Let $\mathcal{V}=\{v_i\}_{i\in I}$ be a (continuous) partial multi-utility representation of $\precsim $,
and let $f$ be a (continuous) Richter-Peleg representation of $\precsim_{|Y} $.
Then it is easily checked that $\mathcal{U}=\{v+\alpha f: v\in \mathcal{V}, \alpha \in \mathbb{Q} , \alpha >0\}$\footnote{Here, if a function $v$ or $f$ is not defined on an element $x$, then we define the sum between $v$ and $f$ on $x$ by $(v+f)(x)=\emptyset$, that is, as not defined. Otherwise, the sum is defined as usual: $(v+f)(x)=v(x)+f(x)$.} is a (continuous) partial Richter-Peleg  multi-utility representation of $\precsim $. The continuity of the functions of $\mathcal{U} $ arises from the continuity of the functions of $\mathcal{V} $ and from the continuity of $f$ (see \cite{Willard}).

This argument serves for the corresponding equivalence under upper/lower semicontinuity too.
\end{proof}

These new partial utilities migth be useful for dealing with random structures (see Schellekens \cite{L1}).  These structures are studied in computation, with several questions unsolved yet.
In the following lines we include some basic definitions related to this topic.

%So now, we could propose an analogous study with the \emph{partial random structure}, which would be the family of all possible partial utilities.

\begin{definition}\label{lakeste}
 Let $(X,\sqsubseteq )$ be a finite partially ordered set with $|X|=n$. We define a \emph{labeling} of the partial order as a function $u\colon (X, \sqsubseteq)\to \{1, ..., n\}$ such that for any $x\sqsubset y$ it holds that $u(x)<u(y)$,  $x,y \in X$.
 \end{definition}

 \begin{definition}\label{repre}
 Let $(X,\sqsubseteq )$ be a finite partially ordered set. The collection $\mathcal{U}=\{u_i\}_{i\in I}$ of all possible labelings is called the \emph{random structure} of the partial order, and it is also denoted by $\mathcal{R}_{\mathcal{L}}(X,\sqsubseteq)$\footnote{Here, $\mathcal{L}$ denotes the set of labels where labeling functions take values. In this paper,   $\mathcal{L}$ will be the set $\{1,...,n\}$, where $n=|X|$. Therefore, we will omit this subscript.}.% if and only if  for any $x\neq y \in X$ it holds that  $x\sqsubseteq y$ if and only if $F(x)<F(y)$, for any $F\in \mathcal{F}$ . Actually, we say that
 \end{definition}

\begin{remark}
\noindent$(1)$
Notice that a labeling is simply a linear order extending the strict preference
and the random structure is simply the set of all such extensions.

\noindent$(2)$
There is a unique correspondence between partial orders and random structures: each of ones defines the other (see \cite{L1}).

\noindent$(3)$ Notice that the concept of random structure implies a Richter-Peleg multi-utility representation.
\end{remark}

%%%%%%%%%%%%%%%%%%%%%%%%%%%%%%%%%%%%%%%%%%%%%%%%%%%%%%%%%%%%%%%%%%%%%%%%%%%%%

%%%%%%%%%%%%%%%%%%%%%%%%%%%%%%%%%%%%%%%%%%%%%%%%%%%%%%%%%%%%%%%%%%%%%%%%%%%%%

%%%%%%%%%%%%%%%%%%%%%%%%%%%%%%%%%%%%%%%%%%%%%%%%%%%%%%%%%%%%%%%%%%%%%%%%%%%%

\begin{example}\label{Elabel}
Let $(X,\sqsubseteq)$ be the partially ordered set defined by $\{x_1\sqsubset  x_2\sqsubset  x_4$, $x_3 \sqsubset x_4\}$. The corresponding random structure is shown in Figure~\ref{figure111}, whereas a partial multi-utility Richter-Peleg representation is shown in Figure~\ref{figureP111}.% $\{x_1\sqsubset  x_2\sqsubset  x_4$, $x_3 \sqsubset x_4\}$, then it defines the   following Hasse's diagram:
\end{example}

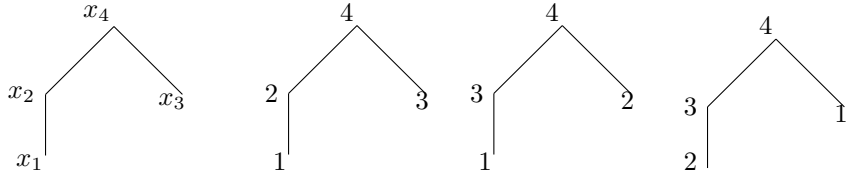
\begin{figure}[htbp]
% \vspace*{0.4cm}
\begin{center}
\begin{tikzpicture}[scale=0.9]
\vspace*{-1.8cm}
% Local definitions
    \draw[] (0,1.1) -- (0,2);
    \draw[] (0,2)--(1,3);
    \draw[] (2,2)--(1,3);
\draw (0,2) node[anchor=east] {\small $x_2$};
\draw (0.12,1) node[anchor=east] {\small $x_1$};
\draw (2.2,1.9) node[anchor=east] {\small $x_3$};
\draw (1.1,3.2) node[anchor=east] {\small $x_4$};
\end{tikzpicture}\quad
\quad \begin{tikzpicture}[scale=0.9]
% Local definitions
    \draw[] (0,1.1) -- (0,2);
    \draw[] (0,2)--(1,3);
    \draw[] (2,2)--(1,3);
\draw (0,2) node[anchor=east] {\small $2$};
\draw (0.12,1) node[anchor=east] {\small $1$};
\draw (2.2,1.9) node[anchor=east] {\small $3$};
\draw (1.1,3.2) node[anchor=east] {\small $4$};
 \draw[] (3,1.1) -- (3,2);
    \draw[] (3,2)--(4,3);
    \draw[] (5,2)--(4,3);
\draw (3,2) node[anchor=east] {\small $3$};
\draw (3.12,1) node[anchor=east] {\small $1$};
\draw (5.2,1.9) node[anchor=east] {\small $2$};
\draw (4.1,3.2) node[anchor=east] {\small $4$};
\end{tikzpicture}\quad
\begin{tikzpicture}[scale=0.9]
  \draw[] (1,-1.7) -- (1,-0.8);
    \draw[] (1,-0.8)--(2,0.2);
    \draw[] (3,-0.8)--(2,0.2);
\draw (1,-0.8) node[anchor=east] {\small $3$};
\draw (1,-1.6) node[anchor=east] {\small $2$};
\draw (3.2,-0.9) node[anchor=east] {\small $1$};
\draw (2.1,0.4) node[anchor=east] {\small $4$};

\end{tikzpicture}
\caption{Partially ordered set with its corresponding labelings.}
\label{figure111}
\end{center}
\end{figure}

\begin{figure}[htbp]
% \vspace*{0.4cm}
\begin{center}
\begin{tikzpicture}[scale=0.9]
\vspace*{-1.8cm}
% Local definitions
    \draw[] (0,1.1) -- (0,2);
    \draw[] (0,2)--(1,3);
    \draw[] (2,2)--(1,3);
\draw (0,2) node[anchor=east] {\small $x_2$};
\draw (0.12,1) node[anchor=east] {\small $x_1$};
\draw (2.2,1.9) node[anchor=east] {\small $x_3$};
\draw (1.1,3.2) node[anchor=east] {\small $x_4$};
\end{tikzpicture}\qquad
\qquad \begin{tikzpicture}[scale=0.9]
% Local definitions
    \draw[] (0,1.1) -- (0,2);
    \draw[] (0,2)--(1,3);
    \draw[] (2,2)--(1,3);
\draw (0,2) node[anchor=east] {\small $2$};
\draw (0.12,1) node[anchor=east] {\small $1$};
\draw (2.2,1.9) node[anchor=east] {\small $\varnothing$};
\draw (1.1,3.2) node[anchor=east] {\small $3$};
 \draw[] (4,1.1) -- (4,2);
    \draw[] (4,2)--(5,3);
    \draw[] (6,2)--(5,3);
\draw (4,2) node[anchor=east] {\small $\varnothing$};
\draw (4.12,1) node[anchor=east] {\small $\varnothing$};
\draw (6.2,1.9) node[anchor=east] {\small $1$};
\draw (5.1,3.2) node[anchor=east] {\small $2$};
\end{tikzpicture}
\caption{Partially ordered set with a  partial Richter-Peleg multi-utility.}
\label{figureP111}
\end{center}
\end{figure}
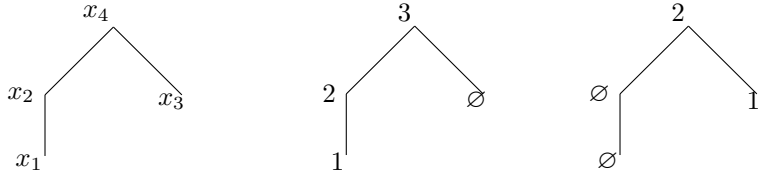

\begin{definition}
Let $\sqsubseteq $ be a finite partial order on $X$. The \emph{Scott topology}
 $\tau_{S}$ is defined by means of the basis
 $\{U_\sqsubseteq (x)\}_{x\in X}= \{\{y\in X \mid x\sqsubseteq y\}\}_{x\in X}$.
\end{definition}
\begin{remark} The Scott topology generated by a partial order is  $T_0$ \cite{L2}.
\end{remark}

The following proposition is well known \cite{L2}.

\begin{proposition}
Let $(X,\sqsubseteq)$ and $(Y, \sqsubseteq')$ be two finite partially ordered sets endowed with the corresponding Scott topologies. Then, a function $f\colon X\to Y$ is continuous if and only if
it is an order-preserving function, that is, if and only if
$x\sqsubseteq y$ implies that  $f(x)\sqsubseteq' f(y)$, for any $x,y\in X$.
\end{proposition}

\begin{corollary}\label{initial}
Let $(X,\sqsubseteq)$  be a finite partially ordered set ($|X|=n$). Then, a function $u\colon X\to \{1,...,n\}$ is a labeling if and only if it is continuous with respect to the corresponding Scott topologies. %Therefore, $\tau_S=\tau_{Init_{u\in \mathcal{R}_{\mathcal{L}}(X,\sqsubseteq)}}$.
\end{corollary}

\begin{remark}
Let $\precsim$ be a preorder on $X$. Notice that if $\precsim$ is partially (Richter-Peleg) multi-utility representable by a family of partial  functions $\mathcal{U}$, then the number of partial functions needed is less or equal than the number of functions needed for an hypothetical (Richter-Peleg) multi-utility representation $\mathcal{U}'$: $|\mathcal{U}|\leq |\mathcal{U}'|. $

%\begin{proof}
Since a (Richter-Peleg) multi-utility representation is also a partial (Richter-Peleg) multi-utility representation, the inequality $|\mathcal{U}|\leq |\mathcal{U}'|$ is trivial. Moreover,  Example~\ref{ex3} shows that in some cases the number of partial functions needed is strictly less than the number of functions needed for an hypothetical (Richter-Peleg) multi-utility representation, so $|\mathcal{U}|< |\mathcal{U}'|. $
%\end{proof}
\end{remark}

\begin{example}\label{ex3}

Let $(X,\sqsubseteq)$ be the partially ordered set defined by $\{x_1\sqsubset  x_2\}$. The corresponding representations are shown in Figure~\ref{figureP3}.
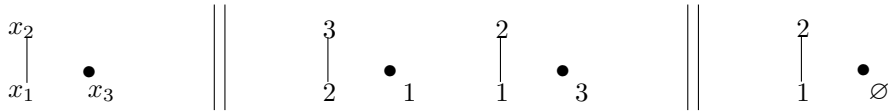
\begin{figure}[htbp]\begin{center}
\begin{tikzpicture}[scale=0.7]
    %\draw[] (1,3) -- (1,2);
  %  \draw[] (0,2.15)--(0,3);
    \draw[] (0,1)--(0,1.85);
        %\draw (-0.2,3.4) node[anchor=east] {\small $(ii)$};
\draw (0.35,2) node[anchor=east] {\small $x_2$};
%\draw (1.2,2) node[anchor=east] {\small $x_2$};
\draw (1.85,0.8) node[anchor=east] {\small $x_3$};
\draw (1.5,1.2) node[anchor=east] {$\bullet$};
\draw (0.35,0.8) node[anchor=east] {\small $x_1$};
\end{tikzpicture}\qquad\quad
\begin{tikzpicture}[scale=0.7]
    %\draw[] (1,3) -- (1,2);
    \draw[] (0.2,0)--(0.2,2);
    \draw[] (0,0)--(0,2);
\end{tikzpicture}\qquad\quad
\begin{tikzpicture}[scale=0.7]
    %\draw[] (1,3) -- (1,2);
  %  \draw[] (0,2.15)--(0,3);
    \draw[] (0,1)--(0,1.85);
\draw (0.35,2) node[anchor=east] {\small $3$};
%\draw (1.2,2) node[anchor=east] {\small $x_2$};
\draw (1.85,0.8) node[anchor=east] {\small $1$};
\draw (1.5,1.2) node[anchor=east] {$\bullet$};
\draw (0.35,0.8) node[anchor=east] {\small $2$};
\end{tikzpicture}\qquad
\begin{tikzpicture}[scale=0.7]
    %\draw[] (1,3) -- (1,2);
  %  \draw[] (0,2.15)--(0,3);
    \draw[] (0,1)--(0,1.85);
\draw (0.35,2) node[anchor=east] {\small $2$};
%\draw (1.2,2) node[anchor=east] {\small $x_2$};
\draw (1.85,0.8) node[anchor=east] {\small $3$};
\draw (1.5,1.2) node[anchor=east] {$\bullet$};
\draw (0.35,0.8) node[anchor=east] {\small $1$};
\end{tikzpicture}\qquad\quad
\begin{tikzpicture}[scale=0.7]
    %\draw[] (1,3) -- (1,2);
    \draw[] (0.2,0)--(0.2,2);
    \draw[] (0,0)--(0,2);
\end{tikzpicture}\qquad\quad
\begin{tikzpicture}[scale=0.7]
    %\draw[] (1,3) -- (1,2);
  %  \draw[] (0,2.15)--(0,3);
    \draw[] (0,1)--(0,1.85);
\draw (0.35,2) node[anchor=east] {\small $2$};
%\draw (1.2,2) node[anchor=east] {\small $x_2$};
\draw (1.85,0.8) node[anchor=east] {\small $\varnothing$};
\draw (1.5,1.2) node[anchor=east] {$\bullet$};
\draw (0.35,0.8) node[anchor=east] {\small $1$};
\end{tikzpicture}
\caption{The random structure and a partial Richter-Peleg multi-utility  representation of a poset.}
\label{figureP3}\end{center}
\end{figure}

\end{example}

\begin{remark}

\noindent$(1)$ In  the example and proposition  before, it is shown that the partial multi-utility representations could be used in order to reduce the cardinal of the family of functions. This reduction may be interesting in other applied fields as computation, for example, dealing with distributed systems in order to reduce the number of clocks \cite{virtual}.
\medskip

\noindent$(2)$ The partial  (Richter-Peleg) multi-utility representations can be studied too through permutations as it was done in \cite{qm}. This study searches properties of the order structure by means of the group properties of the corresponding set of permutations (which is directly defined by the corresponding partial  (Richter-Peleg) multi-utility representation of the order structure).
%\end{remark}
\medskip

\noindent$(3)$ Moreover, partial multi-utility not only reduces the amount of functions needed, but also it allows us to represent some structures that cannot be represented  by the usual multi-utility.
%\begin{proposition}
That is, the set of orderings (and also for the particular case of partial orders) that can be represented through partial (Richter-Peleg) multi-utility is  strictly bigger than this that can be represented through (Richter-Peleg)  multi-utility.
%\end{proposition}
%\begin{proof}
Since a (Richter-Peleg) multi-utility representation is also a partial (Richter-Peleg) multi-utility representation, the inequality  is trivial. Moreover, Example~\ref{notrep}  shows that there are partial Richter-Peleg multi-utility representable orderings that cannot be represented just through Richter-Peleg multi-utility.
%\end{proof}
\end{remark}

\begin{example}\label{notrep}
Let $\precsim$ be the following preorder defined on $X=\mathbb{Q}\times \{0,1\}$:
\[
 (q,i) \prec (p, j)\iff  \left\{
  \begin{array}{ll}
   q< p  &\mbox{ ; } p,q\in X, \forall i,j. \\[4pt]
    q= p  &\mbox{ ; } i=0, j=1.\\[4pt]
  \end{array}\right.
\]
%q\in \mathbb{Q}\times \{0\}, \, p\in \mathbb{Q}\times \{1\}
So, $(p,i)\sim (q,j) $ if and only if $ p=q$ and  $i=j$.

Then, since there is an infinite number of jumps,  $((q,0), (q,1))$ for each $q\in \mathbb{Q}$, it is well known (see Bridges and Mehta \cite{BRME}) that there does not exist a Richter-Peleg utility and, therefore, the preorder $\precsim$ (that actually is a total order) fails to be Richter-Peleg multi-utility representable.

However, we are able to construct a partial Richter-Peleg multi-utility representation by means of  --at least-- a countable number of partial functions:

Let $\phi$ be a bijection from $\mathbb{Q}$ to $\mathbb{N}$. Now, for each $n=\phi(q)\in \mathbb{N}$ we define the following two partial functions on $X$:
\[
u_n((p,i))=  \left\{
  \begin{array}{ll}
   p-1  &\mbox{ ; } p\leq q \text{, } i=0.\\[4pt]
   p  &\mbox{ ; } q<p \text{, } i=0.\\[4pt]
   q &\mbox{ ; } p= q \text{, } i=1.\\[4pt]
  \end{array}\right.\quad
v_n((p,i))=  \left\{
  \begin{array}{ll}
   p  &\mbox{ ; } p< q \text{, } i=1.\\[4pt]
   p+1  &\mbox{ ; } q\leq p \text{, } i=1.\\[4pt]
   q &\mbox{ ; } p= q \text{, } i=0.\\[4pt]
  \end{array}\right.
\]

Moreover, it can be proved that, if $X$ is endowed with the order topology $\tau_{\prec}$, then the partial representation is continuous.
 \end{example}

%{\color{red}
%Example2: A continuously partially representable preorder that fails to be continuously representable.
%}

In the  Example~\ref{E05} it is shown that,
 even not continuous partial orders may be continuously represented through a finite partial Richter-Peleg multi-utility representation. This is impossible with multi-utility, as it was shown in Kaminski \cite{Kam}.
 Furthermore, the partial order of Example~\ref{E05} is connected, and since it is not total, it cannot be continuously multi-utility represented  (see Proposition~5.2  of \cite{Alc}).
 We summarize this idea in the following remark: % proposition:

 %\begin{proposition}
  \begin{remark}
There are partial orders that fail to be continuously (Richter-Peleg) multi-utility representable, but that they can be continuously represented by means of a partial (Richter-Peleg) multi-utility representation.
 %\end{proposition}\begin{proof} 
Since a (Richter-Peleg) multi-utility representation is also a partial (Richter-Peleg) multi-utility representation, then it is clear that any
ordering that is continuous (Richter-Peleg) multi-utility representable it is also continuously representable by means of a partial (Richter-Peleg) multi-utility. The converse is not true, as it is shown in  Example~\ref{E05} (as well as in Example~\ref{counterexample}).
%\end{proof}
\end{remark}

  \begin{example}\label{E05}
Let $(X,\sqsubseteq)$ be the partially ordered set of Example~\ref{Elabel} defined by $\{x_1\sqsubset  x_2\sqsubset  x_4$, $x_3 \sqsubset x_4\}$. The corresponding random structure is shown in Figure~\ref{figureP111}. Now we endow the codomain $\{1,2,3,4\}$ with the Scott topology but,   instead of endowing the set $X$ with the corresponding Scott topology, assume that it is endowed with the  topology $\tau_1 =\{  \emptyset,
\{x_4\}, \{x_2, x_3, x_4\}, \{x_1, x_2, x_4\}, \{x_2, x_4\}, X\}$ or with $\tau_2 =\{  \emptyset,
\{x_4\},$  $\{x_2, x_4\},$ $\{x_3, x_4\},$  $\{x_2, x_3, x_4\},$ $X\}$, which are coarser than the corresponding Scott topology. Notice that, as it is shown in Figure~\ref{figuretopo2}, $\tau_1$ and $\tau_2$ are related to two partial orders that refine the partial order $\sqsubseteq$ of the example.

\begin{figure}[htbp]
% \vspace*{0.4cm}
\begin{center}
\begin{tikzpicture}[scale=0.9]
\vspace*{-1.8cm}
% Local definitions
    \draw[] (0,1.1) -- (0,2);
    \draw[] (0,2)--(1,3);
    \draw[] (2,2)--(1,3);
\draw (0,2) node[anchor=east] {\small $x_2$};
\draw (0.12,1) node[anchor=east] {\small $x_1$};
\draw (2.2,1.9) node[anchor=east] {\small $x_3$};
\draw (1.1,3.2) node[anchor=east] {\small $x_4$};
\end{tikzpicture}\qquad
\qquad \begin{tikzpicture}[scale=0.9]
% Local definitions
    \draw[] (0,1.1) -- (0,2);
    \draw[] (0,2)--(1,3);
    \draw[] (2,2)--(1,3);
\draw (0,2) node[anchor=east] {\small $2$};
\draw (0.12,1) node[anchor=east] {\small $1$};
\draw (2.2,1.9) node[anchor=east] {\small $\varnothing$};
\draw (1.1,3.2) node[anchor=east] {\small $3$};
 \draw[] (4,1.1) -- (4,2);
    \draw[] (4,2)--(5,3);
    \draw[] (6,2)--(5,3);
\draw (4,2) node[anchor=east] {\small $\varnothing$};
\draw (4.12,1) node[anchor=east] {\small $\varnothing$};
\draw (6.2,1.9) node[anchor=east] {\small $1$};
\draw (5.1,3.2) node[anchor=east] {\small $2$};
\end{tikzpicture}
\caption{A continuous partial Richter-Peleg multi-utility representation of a poset.}
\label{figureP111berriz}
\end{center}
\end{figure}
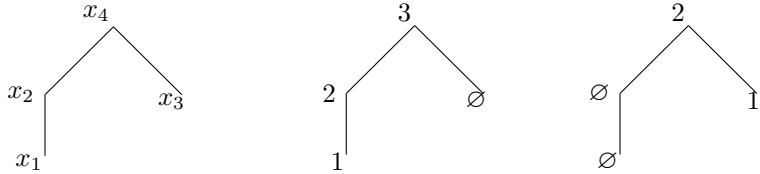

\begin{center}

 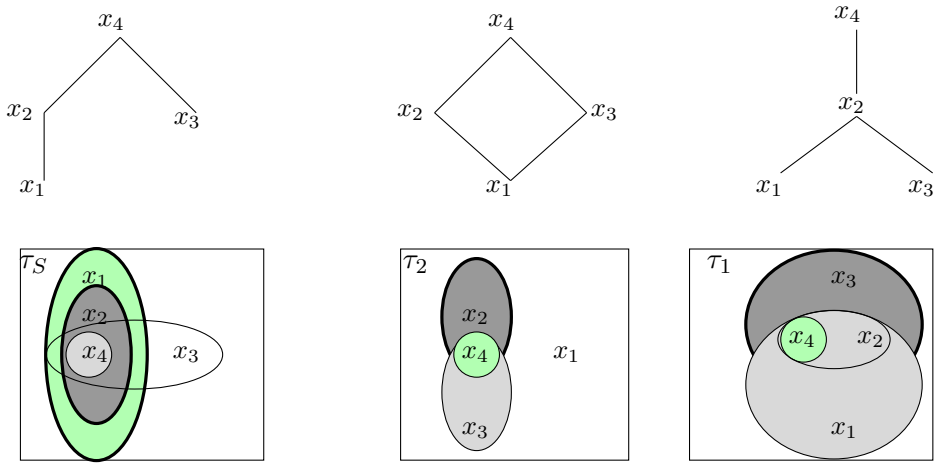
\begin{figure}[htbp]
\begin{tikzpicture}[scale=1]
% Local definitions
    \draw[] (0,1.1) -- (0,2);
    \draw[] (0,2)--(1,3);
    \draw[] (2,2)--(1,3);
\draw (0,2) node[anchor=east] {\small $x_2$};
\draw (0.17,1) node[anchor=east] {\small $x_1$};
\draw (2.2,1.9) node[anchor=east] {\small $x_3$};
\draw (1.2,3.2) node[anchor=east] {\small $x_4$};
%\draw (3.5,2) node[anchor=east] {\large $\leftrightsquigarrow$};
\end{tikzpicture}\qquad \qquad\quad\quad
%$\draw (6,0) node[anchor=east] {\large $\leftrightsquigarrow$}$;
\begin{tikzpicture}[scale=1]
% Local definitions
    \draw[] (1,1.1) -- (0,2);
        \draw[] (1,1.1) -- (2,2);
    \draw[] (0,2)--(1,3);
    \draw[] (2,2)--(1,3);
\draw (0,2) node[anchor=east] {\small $x_2$};
\draw (1.17,1) node[anchor=east] {\small $x_1$};
\draw (2.55,2) node[anchor=east] {\small $x_3$};
\draw (1.2,3.2) node[anchor=east] {\small $x_4$};
\end{tikzpicture}\quad\quad\quad\quad
\begin{tikzpicture}[scale=1]
% Local definitions
    \draw[] (1,1.85)--(2,1.1);
    \draw[] (1,2.15)--(1,3);
    \draw[] (1,1.85)--(0,1.1);
\draw (2.17,0.9) node[anchor=east] {\small $x_3$};
\draw (0.17,0.9) node[anchor=east] {\small $x_1$};
\draw (1.25,2) node[anchor=east] {\small $x_2$};
\draw (1.2,3.2) node[anchor=east] {\small $x_4$};

\end{tikzpicture}

\vspace*{0.5cm}

\begin{tikzpicture}
\begin{scope}[very thick]

\begin{scope}[very thick]

\draw[fill=green!30!white] (-2,0)
ellipse (19pt and 40pt);

\end{scope}
\draw[fill=gray!80!] (-2,0) ellipse (13pt and 26pt);

%\crule{1cm}{1cm}
\end{scope}
\draw[fill=gray!30!white] (-2.1,0) circle (0.3cm);
\draw[thin] (-3,-1.4) rectangle (0.2cm,1.4cm);
\draw[thin] (-1.5,0) ellipse (33pt and 13pt);

\draw (-2.5,1.2) node[anchor=east] {\textbf{$\Large \tau_S$}} ;
\draw (-1.7,0) node[anchor=east] {\small $x_4$} ;
\draw (-0.5,0) node[anchor=east] {\small $x_3$} ;
\draw (-1.7,0.5) node[anchor=east] {\small $x_2$} ;
\draw (-1.7,1) node[anchor=east] {\small $x_1$} ;

%\draw (1.5,0) node[anchor=east] {\large $\leftrightsquigarrow$} ;
%\quad\quad\quad

\begin{scope}[very thick]
\begin{scope}[very thick]

%\draw[fill=green!30!white] (3,0) ellipse (19pt and 40pt);

\end{scope}
\draw[fill=gray!80!] (3,0.5) ellipse (13pt and 22pt);

%\crule{1cm}{1cm}
\end{scope}
\draw[fill=gray!30!] (3,-0.5) ellipse (13pt and 22pt);
\draw[fill=green!30!white] (3,0) circle (0.3cm);
\draw[thin] (2,-1.4) rectangle (5cm,1.4cm);
%\draw[thin] (3.5,0) ellipse (33pt and 13pt);

\draw (2.5,1.2) node[anchor=east] {\textbf{$\Large {\tau}_2$}} ;
\draw (3.3,0) node[anchor=east] {\small $x_4$} ;
\draw (4.5,0) node[anchor=east] {\small $x_1$} ;
\draw (3.3,0.5) node[anchor=east] {\small $x_2$} ;
\draw (3.3,-1) node[anchor=east] {\small $x_3$} ;\quad\qquad\qquad

%\draw (6.5,0) node[anchor=east] {\large $\text{and }$} ;

\begin{scope}[very thick]
\begin{scope}[very thick]

%\draw[fill=green!30!white] (3,0) ellipse (19pt and 40pt);

\end{scope}
\draw[fill=gray!80!] (7.7,0.4) ellipse (33pt and 28pt);

%\crule{1cm}{1cm}
\end{scope}
\draw[fill=gray!30!] (7.7,-0.4) ellipse (33pt and 28pt);
\draw[fill=green!30!white] (7.3,0.2) circle (0.3cm);
\draw[thin] (9,-1.4) rectangle (5.8cm,1.4cm);
\draw[thin] (7.7,0.2) ellipse (21pt and 11pt);

\draw (6.5,1.2) node[anchor=east] {\textbf{$\Large \tau_1$}} ;
\draw (7.6,0.2) node[anchor=east] {\small $x_4$} ;
\draw (8.5,0.2) node[anchor=east] {\small $x_2$} ;
\draw (8.15,1) node[anchor=east] {\small $x_3$} ;
\draw (8.15,-1) node[anchor=east] {\small $x_1$} ;

\end{tikzpicture}\caption{Three partial orders and the corresponding Scott topologies.}\label{figuretopo2}
\end{figure}
 \end{center}

Since the topology on $X$ is coarser than the corresponding Scott topology, by Corollary~\ref{initial} the partially ordered set cannot be continuously Richter-Peleg multi-utility representable. However, it is continuous partial Richter-Peleg multi-utility representable  (with respect to topology $\tau_1$ and also with respect to topology $\tau_2$) through the functions shown in Figure~\ref{figureP111berriz}.

\end{example}

\begin{remark}
A study on the relations between partially ordered finite sets, $T_0$ finite topologies and permutations of the symmetric group is done in the paper \cite{qm} of Estevan et al. entitled \emph{Approximating SP-orders through total preorders: incomparability and transitivity through permutations}.
\end{remark}

%{\color{red}

%\begin{proposition}
%Let $\precsim$ be a preorder on a second countable topological space $(X,\tau)$. If there exits a continuous partial multi-utility $\mathcal{V}$, then there also exists continuous partial Richter-Peleg multi-utility.
%\end{proposition}
%\begin{proof}

%We benefit from a technique in Minguzzi \cite[Theorem 5.5]{Min2}.
%Define $G(\precsim ) = \{(x, y)\in X\times X : x \precsim y\}$ and
%$G_v = \{(x, y)\in X\times X : v(x) \leq v(y)\}$ for each $v\in \mathcal{V}$.
%Then $G(\precsim ) = \bigcap _{v\in \mathcal{V}} G_v$
%and each $G_v$ is closed by continuity of $v$.
%The product space $X\times X$ is second countable (Willard \cite[16E]{Willard}) hence hereditary Lindel\"off (Hocking and Young \cite[Exercise 2-17]{HY}),
%which ensures the existence of a countable family $\mathcal{V}' \subseteq \mathcal{V}$ such that $G(\precsim ) = \bigcap _{v\in \mathcal{V}'} G_v$.
%This means that $\mathcal{V}'$ is a countable continuous partial multi-utility representation of $\precsim $.
%In order to conclude we invoke Proposition~\ref{Pcmurp}.
%\end{proof}

The following lemma of Schmeidler (1971) is well known in literature.

\begin{lemma}\label{lSchm}
Let $\precsim$ be a nontrivial preorder on a connected topological space $(X,\tau)$. If for every $x\in X$ the sets $d(x)$ and $i(x)$ are closed and the sets $l(x)$ and $r(x)$ are open, then the preorder $\precsim$ is total.
\end{lemma}

\begin{proposition}\label{propositiontotal}
Let $\precsim$ be a preorder on a connected topological space $(X,\tau)$ without isolated  points. If there exists a continuous partial Richter-Peleg multi-utility representation $\mathcal{U}=\{u_1, ..., u_n\}$, then $\precsim$ is total on $X$.% and $u_i$ is a  continuous partial utility representation of $\precsim$.
\end{proposition}
\begin{proof}
%It suffices to check that $\precsim $ is total, because in that case any Richter-Peleg representation of $\precsim $ is a utility representation and each $v_i$ is Richter-Peleg representation of $\precsim $ by assumption.
It can be proved that if a preorder $\precsim $ on a topological space $({ X}, \tau )$ has a continuous partial multi-utility representation then both $d(x)$ and $i(x)$ are closed subsets of $ X$ for all $x \in { X}$. This proof is similar to the proof
 of  Theorem 3.1 in Kaminski \cite{Kam} (see also Proposition 5 in Bosi and Herden \cite{BH}) and it is included in the appendix with a lemma.
Therefore, by using Lemma \ref{lSchm}, it suffices to show that under our assumptions,
both $l(x)$ and $r(x)$ are open subsets of $ X$ for all $x \in { X}$.
To prove this fact we observe that, from Definition \ref{partialRPrep},\\
 %\begin{eqnarray*}
$l(x) = \{y \in {X} \mid y \prec x\}=\{y \in {X} \mid v_i(y)< v_i(x), \textnormal{ for all } i\in \{1,...,n\} $  s.t. $v_i$ is defined on both$\}$ $=   \bigcap_{i=1}^n v_i^{-1}((-\infty, v_i(x)))$,\\ %\textnormal{ and }\\
$r(x)=  \{y \in {X} \mid x \prec y\}=\{y \in {X} \mid v_i(x)< v_i(y), \textnormal{ for all } i\in \{1,...,n\}$ s.t. $v_i$ is defined on both$\}$ $= \bigcap_{i=1}^n v_i^{-1}((v_i(x), +\infty))$\\
%\end{eqnarray*}
for each $x \in { X}$. From these equalities and continuity of the functions $v_i$, the conclusion follows immediately.
\end{proof}

\begin{remark}\label{remarkexample}
In the proposition above, if we allow the existence of isolated points, then the statement is false. We see that through the following example:

Let $\precsim$ be a preorder defined on $X=[0,1]\cup \{2\}$ %as follows:
 by $x\precsim y$ if and only if $ x\leq y$ for any $x,y\in [0,1]$, and $ 2 \bowtie x$ for any $ x\in [0,1]$.
%$$x\precsim y \iff x\leq y, \quad x,y\in [0,1], \quad \text{ and } \quad 2 \bowtie x \text{ for any } x\in [0,1].$$
 Now, we endow the set with a topology $\tau$ such that it coincides with the usual topology on $[0,1]$  and such that the open neighbourhoods of $2$ are the same of $0'5$ (that is, $\mathcal{O}_2=\{(0'5-\epsilon, 0'5+\epsilon)\cup \{2\}  \}_{\epsilon >0}$). Therefore, $(X, \tau)$ is a connected topological space.
Under these assumptions, the function $v$ defined by $v(x)=x$ (for any $x\in [0,1]$) and $v(2)=\emptyset$ is a  continuous partial Richter-Peleg multi-utility representation. However, the preorder is not total.

%{\color{blue}
By the way, notice that this preorder fails to be continuously multi-utility representable. To see that, observe that the constant sequence $\{0'5\} $ converges to $2$ so, any function $v$ of a continuous multi-utility must satisfy that $v(2)=v(0'5)$. Therefore, we arrive to the contradiction $0'5\sim 2$. We can argue similarly for semicontinuity. 
%}
\end{remark}

The following example shows another case in which it is not possible to achieve a continuous Richter-Peleg multi-utility, but which can be easily represented through a continuous partial Richter-Peleg multi-utility. In fact, it is proved that the preorder of this example does not admit a continuous  multi-utility representation.
%}
 %Therefore, this example also demonstrates that some results of the literature, for instance,  Theorem 2 in \cite{Ok} and Theorem 3 in \cite{EvO}, are incorrect (see Section~\ref{s6} for further comments).  We include these two theorems before the example (see \cite{Ok} and \cite{EvO} for the literal wording).

%\begin{theorem}
%Let $X$ be a topological space and let $\precsim$ be a near-complete and separable partial
%order on $X$ such that $ \{y ; x\prec y \}$ is open for all $x\in X$. Then there exists a family $\mathcal{U}$ of $n$ upper semicontinuous functions  %$u\colon X \to [0,1]^n$
%such that $\mathcal{U}$ is a multi-utility representaion with %(2) holds and 
%$n=w(X,\precsim )$.
%\end{theorem} 

\begin{example}\label{counterexample}
Let $X$ be the Cartesian product $\mathbb{R}\times \{0,1\}$ (or interpret that as the union of two real lines: $\mathbb{R}_0$ and $\mathbb{R}_1$) endowed with the following preorder:

% \begin{equation}
\[ 
    (x,i)\prec (y,j) \iff \left\{
      \begin{array}{ll}
      x<y, & \text{ and } i=j; \\ \\[4pt]
      x\leq 0  \text{ and } 1<y, & \text{ with } i\neq j; 
  \end{array}\right.
  \]

\begin{figure}[htbp]
\begin{center}
\begin{tikzpicture}[scale=0.8]
\draw[thick] (-1.5,0.2) node[anchor=east] {$\mathbb{R}_0$};
\draw[dashed] (-1.7,0) -- (-1,0);
    \draw[] (-1,0) -- (5,0);
\draw[dashed] (5,0) -- (6,0);

\draw[thick] (1,0.25) node[anchor=east] {\small $0$};

\draw[thick] (1,0) node[anchor=east] { $\bullet$};

%\draw (1.2,2) node[anchor=east] {\small $x_2$};

\draw[thick] (4,0.25) node[anchor=east] {\small $1$};
\draw[thick] (4,0) node[anchor=east] { $\bullet$};
%%%%%%%%%%%%%%%%%%%%%%%%%%%%%%%%%%%%%%%%%%%%

\draw[dashed] (0.85,0) -- (3.85,-2);
\draw[dashed] (0.85,-2) -- (3.85,0);
%%%%%%%%%%%%%%%%%%%%%%%%%%%%%%%%%%%%%%%%%%%

\draw[thick] (-1.5,-1.8) node[anchor=east] {$\mathbb{R}_1$};
\draw[dashed] (-1.7,-2) -- (-1,-2);
    \draw[] (-1,-2) -- (5,-2);
\draw[dashed] (5,-2) -- (6,-2);
\draw[thick] (1,-1.75) node[anchor=east] {\small $0$};
%\draw (1.2,2) node[anchor=east] {\small $x_2$};

\draw[thick] (1,-2) node[anchor=east] { $\bullet$};

\draw[thick] (4,-1.75) node[anchor=east] {\small $1$};
\draw[thick] (4,-2) node[anchor=east] { $\bullet$};
%%%%%%%%%%%%%%%%%%%%%%%%%%%%%%%%%%%%%%%%%%%
%%%%%%%%%%%%%%%%%%%%%%%%%%%%%%%%%%%%%%%%%%%
\end{tikzpicture}
\caption{Preorder defined on $\mathbb{R}\times \{0,1\}$.}
\label{figureCounterExample}
\end{center}
\end{figure}
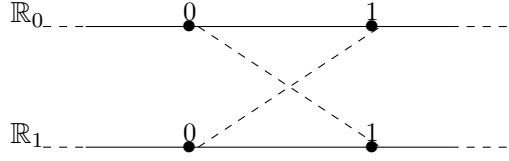

 % \end{equation}
%      %\leq 0  \text{and} 1<y, & \text{ for any } i,j; \\
It is trivial that the width of the preorder is two. 
We endow the set $X$ with the order topology $\tau$ arised from the preorder $\precsim$, except for the  points $(0,0)$ and $(0,1)$, whose open neighbourhood  are as follows:
 \begin{center}
 \begin{tabular}{c}
 $\mathcal{O}_{(0,0)}=\{ (x,0) ; x\in (-\epsilon, +\epsilon )\} \cup \{(y,1); y\in (1, 1+\epsilon)  \}$\\
 $\mathcal{O}_{(0,1)}=\{ (x,1) ; x\in (-\epsilon, +\epsilon )\} \cup \{(y,0); y\in (1, 1+\epsilon)  \}$
 \end{tabular}
\end{center}
 It is easy to check that this topology has a countable basis. Moreover, the contour set $i((0,0))$ is not closed (as well as $i((0,1))$), so the preorder is not regular and therefore, by Theorem~3.1. of   \cite{Kam}, the preorder does not  admit a continuous and finite multi-utility representation.  Furthermore, $l((1'5, 0))$ is not open either (as well as $l((1'5,1))$).

However, the following family of functions is a continuous partial Richter-Peleg multi-utility representation of the preorder:

\[
  u_1((x,i))=\left\{
  \begin{array}{ll}
   x &\mbox{; } i=0\\[4pt]
     \varnothing &\mbox{; } \text{otherwise.} \\[4pt]
  \end{array}\right.\quad
  u_2((x,i))=\left\{
  \begin{array}{ll}
   x &\mbox{; } i=1\\[4pt]
     \varnothing &\mbox{; }\text{otherwise.} \\[4pt]
  \end{array}\right.
  \]
  
  \[
  u_3((x,i))=\left\{
  \begin{array}{ll}
   x &\mbox{; } x\leq 0 \text{ and } i=0  \\[4pt]
   x &\mbox{; } 1<x \text{ and } i=1, \\[4pt]
     \varnothing &\mbox{; } \text{ otherwise.} \\[4pt]
  \end{array}\right.\quad
  u_4((x,i))=\left\{
  \begin{array}{ll}
   x &\mbox{; } x\leq 0 \text{ and } i=1, \\[4pt]
      x &\mbox{; } 1<x \text{ and } i=0, \\[4pt]
     \varnothing &\mbox{; } \text{ otherwise.} \\[4pt]
  \end{array}\right.
\]

Furthermore, it is possible to prove that there is no continuous  Richter-Peleg multi-utility representation of this preorder, even through an infinite number of functions. To see that we can argue by contradiction.

If there was  a continuous  Richter-Peleg multi-utility representation $\{u_i\}_{i\in I}$, since $(0,0)\prec (1+\frac{1}{n}, 1)$ for any $n\in \mathbb{N}$, then each function would satisfy that $u_i((0,0))< u_i((1+\frac{1}{n}, 1))$ for any $n\in \mathbb{N}$. Notice that the sequence $\{(1+\frac{1}{n}, 1)\}_{n\in \mathbb{N}}$ converges to $(0,0)$ (as well as to $(1,1)$) so, applying the limits we have that $ \lim_{n\rightarrow + \infty} \{u_i((1+\frac{1}{n}, 1))\}_{n\in \mathbb{N}} = u_i(1,1)=u_i(0,0)$ for each function of the representation. Since $(0,1)\prec (0'5, 1)\prec (1, 1)$ and $\{u_i\}_{i\in I}$ is a Richter-Peleg multi-utility representation, we conclude that $u_i(0,1)< u_i(0,0)$ for any function of the representation, arriving to the desired contradiction, because $(0,0) \bowtie (0,1)  $.

%BERREGIN BEHEKOA!!!    BERREGIN BEHEKOA !!!!!
%{\color{blue}
In fact, it is also impossible to obtain a continuous multi-utility, even through an infinite number of functions\footnote{As we said, by Theorem~3.1. of   \cite{Kam}, it is known that the preorder does not  admit a continuous and finite multi-utility representation. }. To see that we recover the sequence before and argue again by contradiction. If there was  a continuous   multi-utility representation $\{u_i\}_{i\in I}$, since $(0,0)\prec (1+\frac{1}{n}, 1)$ for any $n\in \mathbb{N}$, then each function would satisfy that $u_i((0,0))\leq u_i((1+\frac{1}{n}, 1))$ for any $n\in \mathbb{N}$. Notice that the sequence $\{(1+\frac{1}{n}, 1)\}_{n\in \mathbb{N}}$ converges to $(0,0)$ (as well as to $(1,1)$) so, applying the limits we have that $ \lim_{n\rightarrow + \infty} \{u_i((1+\frac{1}{n}, 1))\}_{n\in \mathbb{N}} = u_i(1,1)=u_i(0,0)$ for each function of the representation. Therefore, since $u_i(1,1)=u_i(0,0)$ for each function of the representation, we conclude that $(0,0)\precsim (0,1)\precsim (0,0)$, arriving to the desired contradiction, because $(0,0) \bowtie (0,1)  $.
%}

%Since there are now just two functions, $u_1$ and $u_2$,  and for any pair $x\prec y$ it is satisfied that $u_1(x)<u_1(y)$ or $u_2(x)<u_2(y)$, then there exists a subsequence $(1+\frac{1}{n}, 1)_{n\in M\subseteq \mathbb{N}}$ of $(1+\frac{1}{n}, 1)_{n\in \mathbb{N}}$ such that $u_1((0,0))<u_1((1+\frac{1}{n}),1)$ for any $m\in M$. 
%Notice that the subsequence  converges to $(0,0)$ (as well as to $(1,1)$) so, applying the limits we have that $ \lim_{n\rightarrow + \infty} u_i((1+\frac{1}{n}, 1)) = u_i(1,1)=u_i(0,0)$ for each function of the representation.

\end{example}
%}

\section{Partial representability of intransitive relations}\label{s4}

In the present section we study the partial representability of intransitive relations.
If a relation has a multi-utility representation, then it must be transitive (see the introduction). Therefore, multi-utility is not useful dealing with intransitive relations. But this is not the case of partial multi-utility, which allows us to characterize intransitive relations too, as it can be seen in Example~\ref{Esemiz} and Example~\ref{Erpsemi}.

In this section we also include a subsection related to the particular case of semiorders and Scott-Suppes representations. It is well known that semiorders fail to be transitive, and they are usually represented through a Scott-Suppes representation \cite{ales,BRME,Luc,Sco}. In this subsection we introduce the new concept  of a  partial Scott-Suppes representation, which generalizes the classical one. Before this new definition, in the following lines we show some examples of semiorders which are represented through a partial Richter-Peleg multi-utility representation.

\begin{example}\label{Esemiz}
Let $\precsim$ be a semiorder  on $\mathbb{Z}$ defined  by $n\precsim m$ if and only if $ n\leq m+1$. %as follows:
% \[n\precsim m \iff n\leq m+1. \]
Then, the family of functions $\{u,v,w\}$ (see Figure~\ref{figurePicmusemf}) is a partial Richter-Peleg multi-utility representation:
 \[
  u(m)=\left\{
  \begin{array}{ll}
   n &\mbox{; } m=3n \\[4pt]
  n &\mbox{; } m=3n+1 \\[4pt]
   \varnothing &\mbox{; } \text{else} \\[4pt]
  \end{array}\right. \quad
  v(m)=\left\{
  \begin{array}{ll}
   n &\mbox{; } m=3n+1\\[4pt]
  n &\mbox{; } m=3n+2\\[4pt]
   \varnothing &\mbox{; } \text{else} \\[4pt]
  \end{array}\right.
\quad
  w(m)=\left\{
  \begin{array}{ll}
   n &\mbox{, } m=3n+2  \\[4pt]
  n &\mbox{; } m=3n+3  \\[4pt]
   \varnothing &\mbox{; } \text{else,} \\[4pt]
  \end{array}\right.
\]
 where $m$ and $n$ are numbers from $\mathbb{Z}$.
 %\text{ , } n\in \mathbb{Z}\\[4pt]
\end{example}
%\vspace{-0.5cm}
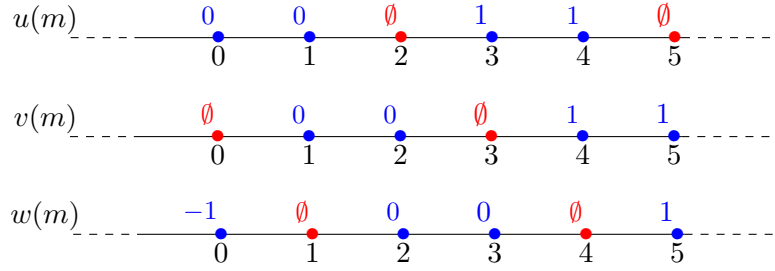
\begin{figure}[htbp]
\begin{center}
\begin{tikzpicture}[scale=1.2]
\draw[thick] (-1.5,0.2) node[anchor=east] {$u(m)$};
\draw[dashed] (-1.7,0) -- (-1,0);
    \draw[] (-1,0) -- (5,0);
\draw[dashed] (5,0) -- (6,0);
 \draw[thick, blue] (0,0.25) node[anchor=east] {\small $0$};
\draw[thick, blue] (1,0.25) node[anchor=east] {\small $0$};
%\draw (1.2,2) node[anchor=east] {\small $x_2$};
\draw[thick, red] (2,0.25) node[anchor=east] {\small $\emptyset$};
\draw[thick, blue] (3,0.25) node[anchor=east] {$1$};
\draw[thick, blue] (4,0.25) node[anchor=east] {\small $1$};
\draw[thick, red] (5,0.25) node[anchor=east] {$\emptyset$};
%%%%%%%%%%%%%%%%%%%%%%%%%%%%%%%%%%%%%%%%%%%%
\draw[thick, blue] (0.12,0) node[anchor=east] {$ \bullet $};
\draw[thick, blue] (1.12,0) node[anchor=east] {$ \bullet$};
%\draw[thick, blue] (0.63,0) node[anchor=east] {$($};

\draw[thick, red] (2.12,0) node[anchor=east] {$ \bullet$};
\draw[thick, blue] (3.12,0) node[anchor=east] {$ \bullet$};

\draw[thick, blue] (4.12,0) node[anchor=east] {$ \bullet$};
\draw[thick, red] (5.12,0) node[anchor=east] {$ \bullet$};
%%%%%%%%%%%%%%%%%%%%%%%%%%%%%%%%%%%%%%%%%%%
\draw[thick] (0.12,-0.2) node[anchor=east] {$0$};
\draw[thick] (1.12,-0.2) node[anchor=east] {$1$};

\draw[thick] (2.12,-0.2) node[anchor=east] {$2$};
\draw[thick] (3.12,-0.2) node[anchor=east] {$3$};

\draw[thick] (4.12,-0.2) node[anchor=east] {$4$};
\draw[thick] (5.12,-0.2) node[anchor=east] {$5$};
\end{tikzpicture}\vspace{0.2cm}
\begin{tikzpicture}[scale=1.2]
\draw[thick] (-1.5,0.2) node[anchor=east] {$v(m)$};
\draw[dashed] (-1.7,0) -- (-1,0);
    \draw[] (-1,0) -- (5,0);
\draw[dashed] (5,0) -- (6,0);
 \draw[thick, red] (0,0.25) node[anchor=east] {\small $\emptyset$};
\draw[thick, blue] (1,0.25) node[anchor=east] {\small $0$};
%\draw (1.2,2) node[anchor=east] {\small $x_2$};
\draw[thick, blue] (2,0.25) node[anchor=east] {\small $0$};
\draw[thick, red] (3,0.25) node[anchor=east] {$\emptyset$};
\draw[thick, blue] (4,0.25) node[anchor=east] {\small $1$};
\draw[thick, blue] (5,0.25) node[anchor=east] {$1$};
%%%%%%%%%%%%%%%%%%%%%%%%%%%%%%%%%%%%%%%%%%%%
\draw[thick, red] (0.12,0) node[anchor=east] {$ \bullet $};
\draw[thick, blue] (1.12,0) node[anchor=east] {$ \bullet$};
%\draw[thick, blue] (0.63,0) node[anchor=east] {$($};

\draw[thick, blue] (2.12,0) node[anchor=east] {$ \bullet$};
\draw[thick, red] (3.12,0) node[anchor=east] {$ \bullet$};

\draw[thick, blue] (4.12,0) node[anchor=east] {$ \bullet$};
\draw[thick, blue] (5.12,0) node[anchor=east] {$ \bullet$};
%%%%%%%%%%%%%%%%%%%%%%%%%%%%%%%%%%%%%%%%%%%
\draw[thick] (0.12,-0.2) node[anchor=east] {$0$};
\draw[thick] (1.12,-0.2) node[anchor=east] {$1$};

\draw[thick] (2.12,-0.2) node[anchor=east] {$2$};
\draw[thick] (3.12,-0.2) node[anchor=east] {$3$};

\draw[thick] (4.12,-0.2) node[anchor=east] {$4$};
\draw[thick] (5.12,-0.2) node[anchor=east] {$5$};
\end{tikzpicture}\vspace{0.2cm}
\begin{tikzpicture}[scale=1.2]
\draw[thick] (-1.5,0.2) node[anchor=east] {$w(m)$};
\draw[dashed] (-1.7,0) -- (-1,0);
    \draw[] (-1,0) -- (5,0);
\draw[dashed] (5,0) -- (6,0);
 \draw[thick, blue] (0,0.25) node[anchor=east] {\small $-1$};
\draw[thick, red] (1,0.25) node[anchor=east] {\small $\emptyset$};
%\draw (1.2,2) node[anchor=east] {\small $x_2$};
\draw[thick, blue] (2,0.25) node[anchor=east] {\small $0$};
\draw[thick, blue] (3,0.25) node[anchor=east] {$0$};
\draw[thick, red] (4,0.25) node[anchor=east] {\small $\emptyset$};
\draw[thick, blue] (5,0.25) node[anchor=east] {$1$};
%%%%%%%%%%%%%%%%%%%%%%%%%%%%%%%%%%%%%%%%%%%%
\draw[thick, blue] (0.12,0) node[anchor=east] {$ \bullet $};
\draw[thick, red] (1.12,0) node[anchor=east] {$ \bullet$};
%\draw[thick, blue] (0.63,0) node[anchor=east] {$($};

\draw[thick, blue] (2.12,0) node[anchor=east] {$ \bullet$};
\draw[thick, blue] (3.12,0) node[anchor=east] {$ \bullet$};

\draw[thick, red] (4.12,0) node[anchor=east] {$ \bullet$};
\draw[thick, blue] (5.12,0) node[anchor=east] {$ \bullet$};
%%%%%%%%%%%%%%%%%%%%%%%%%%%%%%%%%%%%%%%%%%%
\draw[thick] (0.12,-0.2) node[anchor=east] {$0$};
\draw[thick] (1.12,-0.2) node[anchor=east] {$1$};

\draw[thick] (2.12,-0.2) node[anchor=east] {$2$};
\draw[thick] (3.12,-0.2) node[anchor=east] {$3$};

\draw[thick] (4.12,-0.2) node[anchor=east] {$4$};
\draw[thick] (5.12,-0.2) node[anchor=east] {$5$};
\end{tikzpicture}
\caption{A partial Richter-Peleg multi-utility  representation of the semiorder.}
\label{figurePicmusemf}
\end{center}
\end{figure}

\begin{example}\label{Erpsemi}
Let $\precsim$ be a semiorder defined on $\mathbb{R}$ as follows:
\[x\precsim y \iff x\leq y+1. \]
Then, the family of functions %$\{u,v,w\}$
 $\{u_r\}_{r\in [0,2)}$ is a partial Ritcher-Peleg multi-utility representation (see Figure~\ref{figurePicmusem}):
 \[
  u_r(x)=\left\{
  \begin{array}{ll}
   4n &\mbox{; } x\in [4n+r, 4n+r+1) \\[4pt]
     \varnothing &\mbox{; } x\in [4n+r+1, 4n+r+2) \\[4pt]
  4n+2 &\mbox{; } x\in [4n+r+2, 4n+r+3) \\[4pt]
     \varnothing &\mbox{; } x\in [4n+r+3, 4n+r+4) \\[4pt]
  \end{array}\right.
\]
 where $x\in \mathbb{R}$ and $n\in \mathbb{Z}$.
 %\text{ , } n\in \mathbb{Z}\\[4pt]
\end{example}

\begin{figure}[htbp]
\begin{center}
\begin{tikzpicture}[scale=1.3]
\draw[thick] (-1.5,0.2) node[anchor=east] {$u_0(x)$};
\draw[dashed] (-1.7,0) -- (-1,0);
  %  \draw[] (-1,0) -- (5,0);
    \draw[thick, red] (-1,0)--(0,0);
    \draw[thick, blue] (0,0)--(1,0);
    \draw[thick, red] (1,0)--(2,0);
    \draw[thick, blue] (2,0)--(3,0);
    \draw[thick, red] (3,0)--(4,0);
    \draw[thick, blue] (4,0)--(5,0);
\draw[dashed] (5,0) -- (6,0);
 \draw[thick, red] (-0.5,0.25) node[anchor=east] {\small $\emptyset$};
\draw[thick, blue] (0.5,0.25) node[anchor=east] {\small $0$};
%\draw (1.2,2) node[anchor=east] {\small $x_2$};
\draw[thick, red] (1.5,0.25) node[anchor=east] {\small $\emptyset$};
\draw[thick, blue] (2.5,0.25) node[anchor=east] {$2$};
\draw[thick, red] (3.5,0.25) node[anchor=east] {\small $\emptyset$};
\draw[thick, blue] (4.5,0.25) node[anchor=east] {$4$};
%%%%%%%%%%%%%%%%%%%%%%%%%%%%%%%%%%%%%%%%%%%%
\draw[thick, blue] (0.12,0) node[anchor=east] {$[$};
\draw[thick, blue] (1.12,0) node[anchor=east] {$)$};
%\draw[thick, blue] (0.63,0) node[anchor=east] {$($};

\draw[thick, blue] (2.12,0) node[anchor=east] {$[$};
\draw[thick, blue] (3.12,0) node[anchor=east] {$)$};

\draw[thick, blue] (4.12,0) node[anchor=east] {$[$};
\draw[thick, blue] (5.12,0) node[anchor=east] {$)$};
%%%%%%%%%%%%%%%%%%%%%%%%%%%%%%%%%%%%%%%%%%%
\draw[thick] (0.12,-0.2) node[anchor=east] {$0$};
\draw[thick] (1.12,-0.2) node[anchor=east] {$1$};

\draw[thick] (2.12,-0.2) node[anchor=east] {$2$};
\draw[thick] (3.12,-0.2) node[anchor=east] {$3$};

\draw[thick] (4.12,-0.2) node[anchor=east] {$4$};
\draw[thick] (5.12,-0.2) node[anchor=east] {$5$};
\end{tikzpicture}\vspace{0.6cm}

\begin{tikzpicture}[scale=1.3]
\draw[thick] (-1.5,0.2) node[anchor=east] {$u_{0.5}(x)$};
\draw[dashed] (-1.7,0) -- (0,0);
  %  \draw[] (-1,0) -- (5,0);
    \draw[thick, red] (-0.5,0)--(0.5,0);
    \draw[thick, blue] (0.5,0)--(1.5,0);
    \draw[thick, red] (1.5,0)--(2.5,0);
    \draw[thick, blue] (2.5,0)--(3.5,0);
    \draw[thick, red] (3.5,0)--(4.5,0);
    \draw[thick, blue] (4.5,0)--(5.5,0);
\draw[dashed] (5.5,0) -- (6,0);
 \draw[thick, red] (-0.0,0.25) node[anchor=east] {\small $\emptyset$};
\draw[thick, blue] (1,0.25) node[anchor=east] {\small $0$};
%\draw (1.2,2) node[anchor=east] {\small $x_2$};
\draw[thick, red] (2,0.25) node[anchor=east] {\small $\emptyset$};
\draw[thick, blue] (3,0.25) node[anchor=east] {$2$};
\draw[thick, red] (4,0.25) node[anchor=east] {\small $\emptyset$};
\draw[thick, blue] (5,0.25) node[anchor=east] {$4$};
%%%%%%%%%%%%%%%%%%%%%%%%%%%%%%%%%%%%%%%%%%%%
%\draw[thick, red] (-0.38,0) node[anchor=east] {$($};
\draw[thick, blue] (0.62,0) node[anchor=east] {$[$};
\draw[thick, blue] (1.62,0) node[anchor=east] {$)$};
%\draw[thick, blue] (0.63,0) node[anchor=east] {$($};

\draw[thick, blue] (2.62,0) node[anchor=east] {$[$};
\draw[thick, blue] (3.62,0) node[anchor=east] {$)$};

\draw[thick, blue] (4.62,0) node[anchor=east] {$[$};
\draw[thick, blue] (5.62,0) node[anchor=east] {$)$};
%%%%%%%%%%%%%%%%%%%%%%%%%%%%%%%%%%%%%%%%%%%
\draw[thick] (0.12,-0.2) node[anchor=east] {$0$};
\draw[thick] (0.7,-0.2) node[anchor=east] {$0.5$};
\draw[thick] (1.7,-0.2) node[anchor=east] {$1.5$};

\draw[thick] (2.7,-0.2) node[anchor=east] {$2.5$};
\draw[thick] (3.7,-0.2) node[anchor=east] {$3.5$};

\draw[thick] (4.7,-0.2) node[anchor=east] {$4.5$};
\draw[thick] (5.7,-0.2) node[anchor=east] {$5.5$};
\end{tikzpicture}\vspace{0.6cm}

\begin{tikzpicture}[scale=1.3]
\draw[thick] (-1.5,0.2) node[anchor=east] {$u_1(x)$};
\draw[dashed] (-1.7,0) -- (-1,0);
  %  \draw[] (-1,0) -- (5,0);
    \draw[thick, blue] (-1,0)--(0,0);
    \draw[thick, red] (0,0)--(1,0);
    \draw[thick, blue] (1,0)--(2,0);
    \draw[thick, red] (2,0)--(3,0);
    \draw[thick, blue] (3,0)--(4,0);
    \draw[thick, red] (4,0)--(5,0);
\draw[dashed] (5,0) -- (6,0);
 \draw[thick, blue] (-0.5,0.25) node[anchor=east] {\small $0$};
\draw[thick, red] (0.5,0.25) node[anchor=east] {\small $\emptyset$};
%\draw (1.2,2) node[anchor=east] {\small $x_2$};
\draw[thick, blue] (1.5,0.25) node[anchor=east] {\small $2$};
\draw[thick, red] (2.5,0.25) node[anchor=east] {$\emptyset$};
\draw[thick, blue] (3.5,0.25) node[anchor=east] {\small $4$};
\draw[thick, red] (4.5,0.25) node[anchor=east] {$\emptyset$};
%%%%%%%%%%%%%%%%%%%%%%%%%%%%%%%%%%%%%%%%%%%%
\draw[thick, blue] (-0.88,0) node[anchor=east] {$[$};
\draw[thick, blue] (0.12,0) node[anchor=east] {$)$};
\draw[thick, blue] (1.12,0) node[anchor=east] {$[$};
%\draw[thick, red] (0.63,0) node[anchor=east] {$($};

\draw[thick, blue] (2.12,0) node[anchor=east] {$)$};
\draw[thick, blue] (3.12,0) node[anchor=east] {$[$};

\draw[thick, blue] (4.12,0) node[anchor=east] {$)$};
\draw[thick, blue] (5.12,0) node[anchor=east] {$[$};
%%%%%%%%%%%%%%%%%%%%%%%%%%%%%%%%%%%%%%%%%%%
\draw[thick] (0.12,-0.2) node[anchor=east] {$0$};
\draw[thick] (1.12,-0.2) node[anchor=east] {$1$};

\draw[thick] (2.12,-0.2) node[anchor=east] {$2$};
\draw[thick] (3.12,-0.2) node[anchor=east] {$3$};

\draw[thick] (4.12,-0.2) node[anchor=east] {$4$};
\draw[thick] (5.12,-0.2) node[anchor=east] {$5$};
\end{tikzpicture}
\caption{Three partial functions of the partial Richter-Peleg multi-utility  representation $\{u_r\}_{r\in [0,2)}$.}
\label{figurePicmusem}
\end{center}
\end{figure}
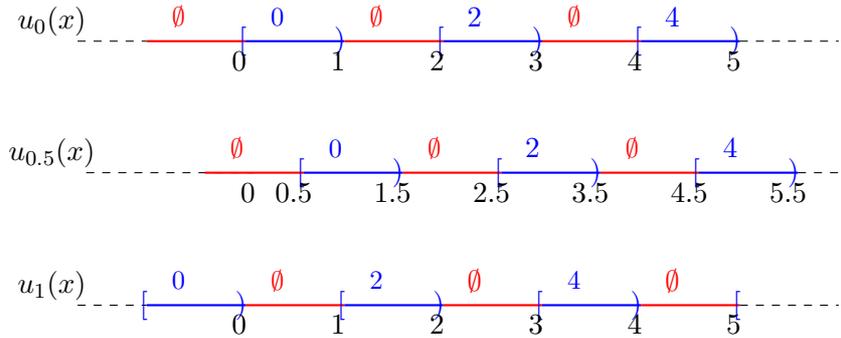

\begin{remark}
Notice that, since regularity of the semiorder is not necessary for this kind of partial multi utility representations, there are semiorders that  cannot be represented by means of a Scott-Suppes representation but that they can be represented through a partial multi-utility Richter-Peleg representation.
\end{remark}

\subsection{Partial Scott-Suppes representations}

%Motivation-Intro...
As it has shown, partial (Richter-Peleg) multi-utility may be useful in order to represent intransitive relations and, in particular, semiorders. % in the Example~\ref{Erpsemi},
However, dealing with these last orderings, to keep the threshold could be a good technique if we want to simplify the representation (e.g. reduce the number of functions of the representation and facilitate their construction), but always trying to represent as much semiorders as possible. For this purpose (simplicity and usefulness), mixing partial representability and the threshold seems a plausible answer: partial SS-representability.

Since we keep the threshold, by means of this partial SS-representation we do not renounce at all to the usual Scott-Suppes  representation,  but now (as it is shown by means of examples) we can achieve some properties (mainly representability or (semi)continuity) in cases for which that was impossible in the usual manner.

%We start from semiorders, instead that from the more general case of interval orders, since its corresponding partial representations its more similar to those that we have seen in the sections before.
In the following lines we introduce two new concepts of partial SS-repre\-sen\-ta\-bi\-li\-ty: one in a `multi-utility' manner (the weak one) and the other one in a `Richter-Peleg' manner (the strict one).

\begin{definition}
Let $\precsim$ be a semiorder defined on a set $X$. Then, we say that the semiorder is \emph{partially  Scott-Suppes representable} (\emph{partially SS-representable} for short) if there exists a family of partial functions $\mathcal{U}$ such that:
\begin{enumerate}
\item[$(i)$] $x\precsim y$ if and only if there exists $u\in\mathcal{U}$ such that $u(x)\leq u(y)+1$ and, $v(x)\leq v(y)+1$ for any $v\in \mathcal{U}$ defined in both $x$ and $y$.

\item[$(ii)$] $x\prec y$ if and only if there exists $u\in\mathcal{U}$ such that $u(x)+1< u(y)$ and, $v(x)\leq v(y)+1$ for any $v\in \mathcal{U}$ defined in both $x$ and $y$.
\end{enumerate}
\end{definition}

\begin{definition}
Let $\precsim$ be a semiorder defined on a set $X$. Then, we say that the semiorder is \emph{partially Richter-Peleg Scott-Suppes representable} (\emph{partially RPSS-representable} for short) if there exists a family of partial functions $\mathcal{U}$ such that:
\begin{enumerate}
\item[$(i)$] $x\precsim y$ if and only if there exists $u\in\mathcal{U}$ such that $u(x)\leq u(y)+1$ and, $v(x)\leq v(y)+1$ for any $v\in \mathcal{U}$ defined in both $x$ and $y$.

\item[$(ii)$] $x\prec y$ if and only if there exists $u\in\mathcal{U}$ such that $u(x)+1< u(y)$ and, $v(x)+1< v(y)$ for any $v\in \mathcal{U}$ defined in both $x$ and $y$.
\end{enumerate}
\end{definition}

Since the functions of a partial RPSS-representation $\mathcal{U}$ are isotonic, that is, for any partial function $u\in \mathcal{U}$ defined in a pair $x,y\in X$ such that $x\prec y$ it must hold that $u(x)+1<u(y)$, it migth seem more successful to obtain a   partial RPSS-representation  instead of a  partial SS-representation. However, notice that both kind of representations totally characterize the order structure and, since the concept of  partial RPSS-representation is more restrictive, there are simple cases (as it is shown in Example~\ref{Eseq}) in
which this representation is less `comfortable' or useful than  the  partial SS-representation.

\begin{example}\label{Eseq}
Let $\precsim$ be a semiorder on $X=\{\frac{1}{n}\}_{n\in \mathbb{N}}\cup \{0\}$ defined as follows:
\[  1\prec \frac{1}{2}\prec \cdots\prec \frac{1}{n}\prec\frac{1}{n+1}\prec\cdots\prec 0 , \quad n\in \mathbb{N}.\]
It is trivial that the semiorder is actually a total order. It is trivial too  that this semiorder fails to be regular and so, it cannot be represented through a SS-representation.\footnote{See  \cite{gurea2,gurea1,gurea3} for more necessary conditions for representability and also for continuity.}

Nevertheless, it is easily partial SS-representable just by means of the following two partial functions:
 \[
  u(x)=\left\{
  \begin{array}{ll}
   2n &\mbox{; } x=\frac{1}{n}, \, n\in \mathbb{N} \\[4pt]
     \varnothing &\mbox{; } x=0 \\[4pt]
  \end{array}\right. \quad
  v(x)=\left\{
  \begin{array}{ll}
   0 &\mbox{; } x\in \{\frac{1}{n}\}_{n\in \mathbb{N}}\\[4pt]
    2 &\mbox{; } x=0 \\[4pt]
  \end{array}\right.\]

\end{example}

The example before proves that, since regularity is not a necessary condition for the partial SS-representability, then there are partial SS-representable semiorders that fail to be SS-representable.

\begin{proposition}
A SS-representable semiorder is also partial (Richter-Peleg) SS-re\-pre\-sen\-ta\-ble. The converse is not true since
there are partial (Richter-Peleg) SS-re\-pre\-sen\-ta\-ble semiorders that fail to be SS-re\-pre\-sen\-ta\-ble.
\end{proposition}

However, as the following proposition shows, it is impossible to obtain a partial RPSS-representation through a finite number of partial functions of a non regular semiorder.

\begin{proposition}
Let $\precsim$ be a non regular semiorder on $X$. If there exists a  partial RPSS-representation $\mathcal{U}$ of the semiorder, then the cardinal of $\mathcal{U}$ is infinite.
\end{proposition}
\begin{proof}
Since it is not regular, we may assume that there is a sequence
$\{x_{n}\}_{n\in \mathbb{N}}$ and an element $x\in X$ such that
$ x_1\prec x_2\prec \cdots\prec x_{n}\prec x_{n+1}\prec\cdots\prec x $,  $n\in \mathbb{N}.$

By contradiction, we suppose that the cardinal of $\mathcal{U}$ is finite. Since $x_n\prec x$ for any $n\in \mathbb{N}$, for any $n\in \mathbb{N}$ there must exist a partial function $u_n\in \mathcal{U}$ such that $u_n(x_n)+1<u_n(x)$. But, since the cardinal of $\mathcal{U}$ is finite, it implies that there must exists a partial function $u$ defined on a infinite number of elements $\{x_{m_k}\}_{m_k\in M\subseteq \mathbb{N}}$  such that
$ x_{m_1}\prec x_{m_2}\prec \cdots\prec x_{m_s}\prec x_{m_{s+1}}\prec\cdots\prec x $, $ s\in \mathbb{N}$, as well as on $x$.
Therefore, we arrive to the desired contradiction, since the existence of this function implies that $u(x)>n$ for any $n\in \mathbb{N}$, what is not possible.
\end{proof}

\begin{remark}
Therefore, partial SS-representations allow us to represent not regular semiorder even through a finite number of functions, whereas they cannot be represented by means of the usual SS-representation.
Not regular semiorders may be represented too by means of a partial RPSS-representation, but in this case an infinite number of partial functions is needed.
\end{remark}

Moreover, the advantages of the partial  (RP)SS-representations can be found dealing with continuity too:

%\begin{proposition}
\begin{remark}
Since any  SS-representation of a semiorder is also a partial (Richter-Peleg) SS-representation,
any continuous  SS-representation of a semiorder is also continuous partial (Richter-Peleg) SS-representation. The converse is not true since
there are continuous partial (Richter-Peleg) SS-representations of semiorders that fail to be continuously SS-representable (see Example~\ref{Esc}).
%\end{proposition}
%\begin{proof}

%The first part is trivial since any  SS-representation of a semiorder is also a partial (Richter-Peleg) SS-representation. The second part  is proved by means of the Example~\ref{Esc}.
%\end{proof}
\end{remark}

  \begin{example}\label{Esc}
Let $\precsim$ be the usual semiorder on $X=\mathbb{R}\setminus (0,0.5]$ defined as:
\[x\precsim y \iff x\leq y+1, \, x,y\in X.\]
If we endow the set $X$ with the corresponding order topology $\tau_<$ of the euclidean order $\leq$ on $X$, it is well known \cite{BRME} (see also \cite{gurea3} for more necessary conditions for continuity) that this semiorder is not continuously representable. However, we can construct a continuous partial  RPSS-representation by means of the following three continuous partial functions:

\[
  u(x)=\left\{
  \begin{array}{ll}
   x &\mbox{; } x\neq 0\\[4pt]
     \varnothing &\mbox{; } x=0 \\[4pt]
  \end{array}\right.\quad
  v(x)=\left\{
  \begin{array}{ll}
   x &\mbox{; } x\notin (0.5, 1]\\[4pt]
     \varnothing &\mbox{; } \text{ else,} \\[4pt]
  \end{array}\right.\quad
  w(x)=\left\{
  \begin{array}{ll}
   0.5 &\mbox{; } x= 0\\[4pt]
   x &\mbox{; } x\in (0.5, 1]\\[4pt]
     \varnothing &\mbox{; } \text{ else.} \\[4pt]
  \end{array}\right.
\]

\end{example}

%{\color{red} Remove interval orders and end.

%\begin{definition}
%Let $\precsim$ be an interval order defined on a set $X$. Then, we say that the interval order is partially i.o. representable if there exist two family of functions $\mathcal{U}$ and $\mathcal{V}$ such that:
%\begin{enumerate}
%\item[$(i)$] $x\precsim y$ if and only if there exist $u\in\mathcal{U}$ and $v\in\mathcal{V}$ such that $u(x)\leq v(y)$ and, $u(x)\leq v(y)$ for any pair $u\in \mathcal{U}$ and $v\in\mathcal{V}$ such that $u$ is defined in  $x$ and $v$ in $y$.

%\item[$(ii)$] $x\prec y$ if and only if there exist $u\in\mathcal{U}$ and $v\in\mathcal{V}$ such that $v(x)<u(y)$ and, $v(x)< u(y)$ for any pair $u\in \mathcal{U}$ and $v\in\mathcal{V}$ such that $u$ is defined in  $y$ and $v$ in $x$.
%\end{enumerate}
%\end{definition}
%}

 \section{Further comments} \label{s6}

In this paper a new concept of representability has been introduced, always trying to match simplicity with usefulness as well as  generalizing the classical ones. Through this partial representability the ordering is totally characterized and, besides, the amount of data needed for doing that is reduced. At the same time, we achieve to represent more orderings than by means of the usual concept of representability. In particular, this may be remarkable in the case of semiorders, where the usual Scott-Suppes representation seems to be too restrictive, since non regular simple semiorders cannot be represented.

Several recent papers have studied the problem of representability \cite{Alc,ales,gurea1,gurea2,EvO,Kam,Min2,Ok}, as well as the corresponding continuity \cite{subm,BH,gurea3,GDOGL}. Although new results and some characterizations have been achieved, some questions are still open.

This new concept of partial representability increases the set of orderings that are now (continuously) representable (by means of partial functions) and so, it opens an analogue study in order to identify that set, as it was done in the papers cited before by using the classical approach.

%{\color{red}
\section*{Appendix}
As we said, a similar proof --but for usual functions, not partial-- of the following lemma was done in \cite{Kam}. 

\begin{lemma}
Let $\precsim $ be a preorder without isolated points defined on a topological space $({ X}, \tau )$. If there exists  a continuous partial multi-utility representation $\{u_1,..., u_n\}$, then both $d(x)$ and $i(x)$ are closed subsets of $ X$ for all $x \in { X}$. 
\end{lemma}
\begin{proof}
First, notice that $d(x)=\{y\in X ; y\precsim x\}$ coincides with the set $\{y\in X ; u_k(y)\leq u_k(x), \forall u_k \text{ defined in both}\}$. 

Consider an arbitrary net $(y_i)_{i\in I}$ convergent in $X$, such that $y_i\in d(x)$ for any $i\in I$.  Let $g$ be the limit of this net: $(y_i)_{i\in I} \longrightarrow g$. In order to show that $d(x)$ is closed it suffices to prove that $g\in d(x)$.

By definition there is no isolated points. Therefore, since the number of function is finite, there must be --at least-- a function $u_l$ defined in all the points of a subnet $(y_j)_{j\in J\subseteq I}$ (otherwise, there exists an index $i_0\in I$ such that there is no function defined in $y_i$ for any $i>i_0$, arriving to a contradiction). It is well known that the subnet converges to the same point of the net, that is,  $(y_j)_{j\in J\subseteq I}\longrightarrow g$. 
Since  functions $u_k$ --in particular $u_l$-- are continuous,  $\{u_l (y_j)\}_{j\in J\subseteq I}\longrightarrow u_l(g)$. But $u_l(y_j)\leq u_l(x)$ for any $j\in J$ so, we deduce that $u_l(g)\leq u_l(x)$. Thus, $g\in d(x)$. Therefore, $d(x) $ is closed.

Analogous proof is done for $i(x)$.
 \end{proof}
%}
 \section*{Acknowledgements}

 Asier Estevan acknowledges financial support 
 from the Ministry of Economy and Competitiveness of Spain under grants MTM2012-37894-C02-02, MTM2015-63608-P and ECO2015-65031 as well as from the Basque Government under grant IT974-16. Gianni Bosi and Magal\`\i\ E. Zuanon acknowledge financial support from the Istituto Nazionale di Alta Matematica  $\lq\lq$F. Severi" (Italy).

%\end{acknowledgements}


\section*{References}

\begin{thebibliography}{}


\bibitem{Alc} J.C.R. Alcantud, G. Bosi, M. Zuanon, Richter-Peleg multi-utility representations of preorders, \emph{Theory Dec.}   {\bf 80} (2016), 443-450.

% \bibitem{Alc2} J.C.R. Alcantud, G. Bosi, M. Zuanon, Representations of preorders by strong multi-objective functions, \emph{MPRA} No. \textbf{52329} \rm (2013).

\bibitem{ales} F. Aleskerov, D. Bouyssou, B. Monjardet, Utility maximization, choice and preference (second edition), Springer, Berlin, 2007.

 %\bibitem{Bos} G. Bosi, M.J. Campi\'on, J.C. Candeal and E. Indur\'ain, Interval-valued representability of qualitative data: the continuous case, \emph{Internat. J. Uncertain. Fuzziness Knowledge-Based Systems} \textbf{15 (3)} \rm (2007) 299-319.

 \bibitem{Bio} G. Bosi, J.C. Candeal and E. Indur\'ain, Continuous representability of interval orders and biorders, \emph{J. Math. Psych.} \textbf{51} \rm (2007) 122-125.



 \bibitem{subm} G. Bosi,  A. Estevan, J. Guti\'errez Garc\'{\i}a and E. Indur\'ain, Continuous representability of interval orders, the topological compatibility setting,  \emph{Internat. J. Uncertain. Fuzziness Knowledge-Based Systems} \textbf{23 (3)}\rm (2015) 345-365.

 %\bibitem{BH2012} Bosi, G., Herden, G., Continuous multi-utility representations of  preorders, {\em Journal of Mathematical Economics} {\bf 48} (2012), 212-218.

 \bibitem{BH} G. Bosi, G.  Herden,  On continuous multi-utility representations of semi-closed and closed preorders, {\em Mathematical Social Sciences} \textbf{79} (2016), 20-29.

 % \bibitem{frontiers}  D. Bouyssou and T. Marchant, Biorders with frontiers, \emph{Order} \textbf{28 (1)} \rm (2011) 53-87.


 \bibitem{BRME} D.S. Bridges and G.B. Mehta, \textit{Representations
of Preference Orderings},
 Berlin-Heidelberg-New York: Springer-Verlag, 1995.



% \bibitem {Campi} M.J. Campi\'on, J.C. Candeal, E. Indur\'ain and M. Zudaire, Continuous representability of semiorders, \emph{J. Math. Psych.} \textbf{
%52} \rm (2008) 48-54.


\bibitem{gurea2}  J.C. Candeal,   A. Estevan, J.  Guti\'errez-Garc\'ia,  E.  Indur\'ain,  Semiorders with separability properties, \emph{J. Math. Psych.}  \textbf{56} (2012), 444-451.




 %\bibitem{CIZ}
%J.C. Candeal, E. Indur\'ain and M. Zudaire,
%Continuous representability of interval orders,
 %\emph{Appl. Gen. Topol.} \textbf{5 (2)} \rm (2004) 213-230.

\bibitem{L2} B. A. Davey, H. A. Priestley, \emph{Introduction to lattices and order (second edition)}. Cambridge. Cambridge University Press, 2002. ISBN-10: 0521784514.
 % \bibitem{chateau} A. Chateauneuf, Continuous representation of a preference preference relation on a connected topological space,  \emph{J. Math. Econom.} \textbf{16} \rm (1987) 139-146.


 \bibitem{de} G. Debreu, Representation of a preference
ordering by a numerical function, In R. Thrall, C. Coombs and R.
Davies (Eds.), \textit{Decision processes}, New York: John Wiley, 1954.

 \bibitem{Debr} G. Debreu, Continuity Properties of Paretian Utility,
\emph{Internat. Econom. Rev.} \textbf{5} (1964) 285-293.



 \bibitem{DS} \mbox{A. Estevan, Some results on biordered structures, in particular distributed systems,} \mbox{ \emph{Journal of Mathematical Psychology} (2016) dx.doi.org/10.1016/j.jmp.2016.09.001.} 


 \bibitem{GDOGL} A. Estevan, Generalized Debreu's Open Gap Lemma and continuous representability of biorders, \emph{Order} \textbf{33, (2)} (2016) 213-229. 
 



  \bibitem{gurea1} A. Estevan, J.  Guti\'errez Garc\'ia,  E. Indur\'ain, Numerical Representation of Semiorders. \emph{Order} \textbf{30}  (2013) 455-462.


  \bibitem{gurea3}  A. Estevan, J.  Guti\'errez Garc\'ia,  E. Indur\'ain,  Further results on the continuous representability of semiorders. \emph{Internat. J. Uncertain. Fuzziness Knowledge-Based Systems} \textbf{ 21,  (5)} (2013) 675-694.


\bibitem{qm} A. Estevan, O. Valero, M. Schellekends, Approximating SP-orders through total preorders:  incomparability and transitivity through permutations. \emph{Quaestiones Mathematicae}. Accepted.

\bibitem{EvO} O.  Evren,   E.A. Ok, On the multi-utility representation of preference relations,
{\em Journal of Mathematical Economics} {\bf 47} (2011),  554-563.

 \bibitem{fid} C. Fidge, Logical time in distributed computing systems, \emph{IEEE Computer}  \rm (1991) 28-33.


%\bibitem{Fis} P.C. Fishburn, \emph{Utility theory for decision-making}\rm, New York: Wiley, 1970.

\bibitem{Fis2} P.C. Fishburn, Intransitive indifference in preference
theory: a survey, \emph{Oper. Res.} \textbf{18 (2)} \rm (1970) 207-228.

\bibitem{HY} J. G. Hocking and G. S. Young, Topology. Dover (1988).
 %\bibitem{[8]} P.C. Fishburn, Intransitive indifference with unequal indifference intervals, \textit{J. Math. Psychol.} \textbf{7} \rm (1970) 144-149.

\bibitem{Kam} B. Kaminski, On quasi-orderings and multi-objective functions, {\em European Journal of Operational Research} {\bf 177} (2007), 1591-1598.

% \bibitem{kop} H. Kopetz, W. Ochsenreiter, Clock synchronization in distributed real-time systems, \emph{IEEE Transactions on Computers} \textbf{C-66 (8)} \rm (1987) 933-940.



\bibitem{Lamport} L. Lamport, Time, clocks and the ordering of events in a distributed system, \emph{ACM Commun. Comput. Algebra} \textbf{21 (7)} \rm (1978)
558-565.




\bibitem{Levin} V. L. Levin, Functionally closed preorders and strong stochastic dominance, \emph{Soviet
Math. Doklady} \textbf{32} \rm (1985), 22-26.


\bibitem{Levin2} V. L. Levin, The Monge-Kantorovich problems and stochastic preference relation, \emph{Advances in
Mathematical Economics}, \textbf{3} \rm (2001), 97-124.


\bibitem{Luc} R.D. Luce, Semiorders and a theory of utility discrimination, \emph{Econometrica} \textbf{24} \rm (1956)
178-191.


\bibitem{virtual} F. Mattern, Virtual time and global satates of distributed systems, \emph{Proceedings of the International Workshop on Parallel and Distributed Algorithms}.

%\bibitem{Min1} Minguzzi, E., Topological conditions for the representation of preorders by continuous utilities, {\em Applied General Topology} {\bf 13} (2012), 81-89.

\bibitem{Min2} E. Minguzzi,  Normally Preordered Spaces and Utilities, {\em Order} {\bf 30} (2013), 137-150.

 % \bibitem{Mo} B. Monjardet, Axiomatiques et propri\'et\'es des quasi-ordres, \emph{Math. Sci. Hum.} \textbf{63} \rm (1978) 51-82.

 \bibitem{Naka} Y. Nakamura, Real interval representations, \emph{J. Math. Psych.} \textbf{46} \rm (2002) 140-177.

 \bibitem{Nish} H. Nishimura, E. A. Ok, Representation of an incomplete and nontransitive preference relation, Journal of Economic Theory,  \textbf{166} (2016),  164-185.

\bibitem{Ok}  E. A. Ok, Utility representation of an incomplete preference relation, {\em Journal of  Economic
Theory} {\bf 104} (2002), 429-449.

\bibitem{Panan} P. Panangaden, Causality in Physycs and computation, \emph{Theorical Computer Science} \textbf{546} \rm (2014)
10-16.


 \bibitem{Peleg} B. Peleg, Utility functions for partially ordered topological spaces, \emph{Econometrica}  \rm (1996)
49-56.

 \bibitem{ray} M. Raynal, M. Singhal, Logical time: capturing causality in  distributed  systems, \emph{IEEE Computer}  \rm (1991) 28-33.

  \bibitem{Richter} M. Richter, Revealed preference theory,  \emph{Econometrica} \textbf{34} \rm (1966)
635-645.

 \bibitem{L1} M. Schellekens, \emph{A modular calculus for the average cost of data structuring}. Springer. 2008.
SBN 978-0-387-73384-5.

  \bibitem{Sco} D. Scott, P. Suppes, Foundational aspects of theories of measurement, \emph{J. Symbolic
Logic} \textbf{23} \rm (1958) 113-128.

 % \bibitem{torres} F. J. Torres-Rojas, Mustaque Ahamad, Plausible clocks: constant size logical clocks for distributed systems, \emph{Distributed Computing} \textbf{12} \rm (1999) 179-195.

%\bibitem{Tve} A. Tversky, Intransitivity of preferences, \emph{Psychological Review} \textbf{76} \rm (1969) 31-48.



\bibitem{Willard} S. Willard,
General Topology. Dover (2004).

\end{thebibliography}
\end{document}